\documentclass[a4paper,11pt]{amsart}
\usepackage{amsmath,amssymb,mathabx,amsfonts,graphicx}
\usepackage[mathscr]{euscript}
\usepackage[all]{xy}
\usepackage{enumitem}
\usepackage{manfnt}
\usepackage{inputenc}
\usepackage{xcolor}
\setenumerate[1]{label={\textup{(\arabic*)}}}
    \theoremstyle{plain}
\newtheorem{theorem}{Theorem}[section]
\newtheorem*{theorema}{Theorem A}
\newtheorem*{theoremb}{Theorem B}
\newtheorem{fact}{Fact}[section]
\newtheorem{lemma}[theorem]{Lemma}
\newtheorem{proposition}[theorem]{Proposition}
\newtheorem{corollary}[theorem]{Corollary}

\newtheorem{definition}[theorem]{Definition}
\newtheorem{definitions}[theorem]{Definitions}

\newtheorem{example}[theorem]{Example}
\newtheorem*{notation}{Notation}
\theoremstyle{definition}
\newtheorem{remark}[theorem]{Remark}
\newtheorem{remarks}[theorem]{Remarks}

\newcommand\<[1]{\left\langle\, #1\,\right\rangle}
\newcommand\bang[1]{\big\langle\, #1\,\big\rangle}
\newcommand\norm[1]{\lVert #1 \rVert}
\newcommand\bnorm[1]{\big\lVert #1 \big\rVert}

\newcommand{\pp}{p^\prime}
\newcommand{\Cf}{\ensuremath \mbox{\large${\mathbf{1}}$}}

\newcommand{\hh}{\ensuremath \mbox{\tiny $H$}}	
\newcommand{\h}{\ensuremath \mathbb{H}}

\newcommand{\mbld}{\ensuremath
{\ll(L^2(G))}}

\newcommand{\G}{\ensuremath \mbox{\tiny $G$}}	
\newcommand{\mgh}{\ensuremath  \mathrm{m}_{\G/\mbox{\tiny $H$}}}
\newcommand{\dmgh}{\ensuremath  \, \mathrm{dm}_{\G/\mbox{\tiny $H$}}}
\newcommand{\dmh}{\ensuremath  \, \mathrm{dm}_{\mbox{\tiny $H$}}}
\newcommand{\dmg}{\ensuremath  \, \mathrm{dm}_{\G}}
\newcommand{\mg}{\ensuremath  \, \mathrm{m}_{\G}}
\newcommand{\dmu}{\ensuremath \, \mathrm{d}\mu}

\newcommand{\mbx}{\ensuremath {\ll(X)}}
\newcommand\Fix[1]{\ensuremath \ker(#1-I)}
\newcommand\restr[2]{{% we make the whole thing an ordinary symbol
  \left.\kern-\nulldelimiterspace % automatically resize the bar with \right
  #1 % the function
  \vphantom{|} % pretend it's a little taller at normalº size
  \right|_{_{#2}} % this is the delimiter
  }}

\newcommand{\R}{\mathbb{R}}
\newcommand{\C}{\mathbb{C}}
\newcommand{\T}{\mathbb{T}}
\newcommand{\Q}{\mathbb{Q}}
\newcommand{\N}{\mathbb{N}}

\newcommand{\Z}{\mathbb{Z}}

\DeclareMathOperator{\supp}{supp}

\renewcommand{\ll}{\mathcal{L}}

\renewcommand{\emptyset}{\varnothing}

\DeclareMathOperator{\range}{Ran}

\title{Ergodic properties of convolution operators}
\author{Jorge Galindo \and Enrique Jord\'a}

\address{\noindent Jorge Galindo, Instituto Universitario de Matem\'aticas y
Aplicaciones (IMAC)\\ Universidad Jaume I, E-12071, Cas\-tell\'on,
Spain. \hfill\break \noindent E-mail: {\tt jgalindo@mat.uji.es}}
\address{\noindent Enrique Jord\'a, EPS Alcoy, Instituto Universitario de Matem\'atica Pura y Aplicada IUMPA,
Universitat Polit\`ecnica de Val\`encia, Plaza Ferr\'andiz y Carbonell s/n
E-03801 Alcoy, Spain, \hfill\break \noindent E-mail: {\tt ejorda@mat.upv.es}}


\date{\today}

\begin{document}
\begin{abstract}
Let $G$ be a locally compact group and $\mu$ be a measure on $G$.
In this paper we find conditions for the convolution   operators $\lambda_p(\mu)$,  defined on $L^p(G)$ and given by convolution by $\mu$,  to be mean ergodic and uniformly mean ergodic. The ergodic properties of the operators $\lambda_p(\mu)$ are related to  the ergodic properties of the measure $\mu$ as well.
\end{abstract}
\maketitle
\section{Introduction}
Von Neumann's mean ergodic theorem proves that the sequence $(T_{[n]})_n$ of  averages of the first $n$  powers of a unitary operator $T$ (its  Ces\`aro  means) always converges to  the projection onto the subspace of fixed vectors of $T$.  After the appearance of von Neumann's mean ergodic theorem, and its Birkhoff's contemporary, the pointwise ergodic theorem, the convergence of averages of measure preserving transforms in diverse  settings and senses rapidly became an important part  of Ergodic Theory, see \cite{dunfschw,kren85} for more details. In this development, an operator   on a Banach space  whose sequence   of Ces\`aro means converges      in the strong operator topology came to be known as  \emph{mean ergodic}. When convergence holds   in the operator norm, the coined term was \emph{uniformly mean ergodic}.

Our approach to ergodicity can be traced back to \cite{yosi38} where Yosida  used weak clustering of  orbits to characterize mean ergodicity of operators on Banach spaces.
%The impact of structural properties   on mean ergodicity was studied by   Fonf, Lin and Wojstaycyck \cite{FLW}. They in particular showed  that  power bounded operators which are not mean ergodic can be constructed in   every non reflexive Banach space with Schauder basis.
 Our contemporary antecedent is to be found in  the work of    Bonet and Doma\'nski \cite{bd1,bd2} that launched  a systematic research on  ergodicity  from an operator theoretic point of view.
 Composition and  multiplication operators in spaces of holomorphic functions have since been   studied \cite{bgjj1,bgjj2,bd1,bd2} with a recent attention to composition operators  acting on spaces of smooth real functions  \cite{fgj,kalmes}.
In this paper, we address the ergodic properties of  those operators on the Banach space $L^p(G)$ of $p$-integrable functions on a locally compact group $G$ that are given by convolution by a fixed measure on $G$.  We will refer to these operators simply as convolution operators.

 When  $G$ is sigma-compact and $\mu$ is a probability measure, the operator  given by convolution  with $\mu$ is a Markov operator, having the Haar measure $\mg$
as sigma-finite invariant measure;  the Markov chain
generated by $\mu$ with state space $G$ is a random walk on $G$.
       If $G$ is compact or Abelian, conditions for ergodicity and mixing of the random walk
are well known, in terms of the algebraic properties of $\mu$ (the It\^o-Kawada and Choquet-Deny theorems). See Remark \ref{ebc} for some more information on this approach.
%We should notice here that,   in the Harmonic Analysis literature, the term convolution operator often refers to a more general class of operators,  but we will restrict  here  to those defined by bounded measures.
%They  constitute one of the main objects in Abstract Harmonic Analysis as corresponds to an object that  brings  together the group operation and the invariant  measure of a locally  compact grou

%In their several presentations, convolution operators   are at the core of Harmonic Analysis.
 Among the most recent papers dealing with convolution operators from a point of view related to  ours, we can mention the monograph by Derighetti \cite{deri11}, with special emphasis on  their  restriction and extension to  and from closed subgroups,   and the papers by Neufang, Salmi, Skalski and Spronk \cite{neufsalmiskalspro}, with focus on the wider frame of quantum groups, and by Mustafayev  \cite{must19,must19jfaa}, where   ergodicity for multipliers on Banach algebras  and     convolution operators on locally compact Abelian groups are studied.

\subsection{Outline and summary of results}
Our approach leads to characterizations of    mean  ergodicity  and uniform mean ergodicity for  convolution operators defined by several classes of measures. These characterizations  can also be used to compare the convergence of averages of convolution powers of measures in the vague topology (mean ergodicity of measures) with  mean ergodicity of the corresponding convolution operators. As  it turns out, while   mean ergodicity of measures  often implies    ergodicity  of their convolution operators, the converse does not hold even for weak versions of ergodicity.

We next outline  our results, showcasing the case of Abelian groups. Our methods however are not intrinsically commutative and most of our results are stated for more general classes of groups. We also  have also tried to determine the limits of what  can be expected beyond our results through  a series of examples, both commutative and noncommutative, that appear, mostly, in Section \ref{sec:trac}.

\begin{theorema}[Theorem \ref{hilbert}, Theorem \ref{pb=me:amen}, Theorem \ref{lambda1}] Let $G$ be a locally compact  group and $\mu\in M(G)$.
  Consider the following conditions on the convolution operator $\lambda_p(\mu)$ (see Section \ref{sec:pre} for undefined notation):
  \begin{itemize}
\item[(a)] $\lambda_p(\mu)$ is mean ergodic.
\item[(b)]  $\lambda_p(\mu)$ is  power bounded.
\item[(c)] $\|\mu\|\leq 1 $.
\end{itemize}
These conditions can be related as follows.
\begin{enumerate}
  \item If $G$ is Abelian and $p=2$, then  (a) and (b) are equivalent.
\item If $G$ is amenable and  $\mu$ is positive,   then all three conditions are equivalent, for any $1<p<\infty$.
    \item If $G$ is compact, $\mu$ positive and $p=1$, then all  three conditions are equivalent as well.
        \end{enumerate}
\end{theorema}
The last statement of Theorem A actually contains all what needs to be known about the case $p=1$, for  $\lambda_1(\mu)$ cannot be mean ergodic unless either $\mu$ is a probability measure and the support of $\mu$ is contained in a compact subgroup of $G$, see Theorem \ref{lambda1}, or $\|\mu\|<1$. In this last case we have even that $(\lambda(\mu^n))_n$ is norm convergent to 0.

\begin{theoremb}[Theorem \ref{hilbert:ume}, Theorem \ref{reflexive}, Theorem \ref{nonreflexive}]
  Let $G$ be a locally compact  Abelian group and let $\mu\in M(G)$. Consider the following conditions on the convolution operator $\lambda_p(\mu)$:
  \begin{itemize}
    \item[(a)] $\lambda_p(\mu)$ is uniformly mean ergodic.
        \item[(b)] $\norm{\lambda_p(\mu)}\leq 1$ and $1$ is not  an accumulation point of the spectrum of $\lambda_p(\mu)$.
  \end{itemize}
Conditions (a) and (b) are then equivalent when any of the following conditions hold:
  \begin{enumerate}
  \item  $p=2$.
  \item  $\mu$ is positive and $1<p<\infty$.
  \item  $G$ is compact, $\mu$ positive  and  $p=1$ or $p=\infty$. In either case  $\lambda_p(\mu)$ is uniformly mean ergodic if and only if $\lambda_\infty(\mu)$ is mean ergodic.
  \end{enumerate}
\end{theoremb}
In Theorem B, commutativity is only needed to make sure the operator $\lambda_2(\mu)$ is normal. Hence the theorem is also true under this weaker assumption.

 Many of our results on ergodicity or mean ergodicity of $\lambda_p(\mu)$ depend on  conditions on the ambient group $G$. To make them depend, as it is naturally expected,  on the subgroup  $H_\mu$  generated by the support of $\mu$, we have had to relate the ergodic behaviour of $\lambda_p(\mu)$ as an operator on $L^p(G)$ with its behaviour as  an operator on $L^p(H_\mu)$. Since this is a  technically intrincate issue, we have decided to deal with it in an Appendix at the end of the paper.

\section{Preliminaries}\label{sec:pre}
In this section we gather
  the basic definitions and basic facts  around our two main subjects: ergodicity of operators and convolution operators.
 \subsection{Ergodic Operators}
\begin{notation}
  If $T\in\mbx$ is a bounded linear operator on a Banach space $X$, $T_{[n]}$ will denote the Ces\`aro means:
\[ T_{[n]}=\frac{1}{n}\sum_{k=1}^n T^k.\]
\end{notation}

\begin{definition}
We say that a bounded linear operator $T\in \mbx$ is:
 \begin{enumerate}  \item \emph{Weakly mean Ergodic} if there is $P\in \mbx$ such that $\lim_{n\to \infty}  T_{[n]}=P$ in the weak operator topology.
  \item \emph{Mean Ergodic} if there is $P\in \mbx$ such that $\lim_{n\to \infty}  T_{[n]}=P$ in the strong operator topology.
 \item \emph{Uniformly Mean Ergodic} if   there is $P\in \mbx$ such that $\lim_{n\to \infty}  T_{[n]}=P$ in  the  operator norm.
 \end{enumerate}
\end{definition}
\begin{notation}\label{def:pt}
  If $T\in\mbx$ is a bounded linear operator on a Banach space $X$ for which $\Fix{T}$ is a complemented subspace  we will denote by    $P_T$ the projection operator onto $\Fix{T}$.
\end{notation}
The following two basic facts on  mean ergodicity of operators can  be found in or deduced from Section 8.4 of \cite{eisnetal15}.
\begin{theorem}
\label{basic:me}
Let $T$ be a bounded linear operator on a Banach space $X$ such that $\norm{T^n}/n$ converges to 0 in the strong operator topology. Then
\begin{enumerate}
  \item  $T$ is mean ergodic if and only if $X=\Fix{T}\oplus \overline{\range{I-T}}$.
  \item If $T$ is mean ergodic, then $\lim_n T_{[n]}=P_T$ and $P_T T=P_T=T P_T$.
\end{enumerate}
\end{theorem}
We next state a general version of Yosida's mean ergodic Theorem. We include a proof for the sake of completeness. It does not include new ideas but those of Yosida. See \cite{yosi38} for the original proof and \cite[Theorem 2.4]{abr}, \cite[Theorem 8.22]{eisnetal15} or \cite[Theorem 1.3]{Petersen} for similar statements.
\begin{theorem}
\label{weakyosida}
Let $X$ be a Banach space and let $\tau$ be a locally convex topology on $X$  compatible with $\sigma(X,X^*)$. Then there is a bounded linear operator $P$ on $X$ such that $(T_{[n]}x)_n$ is $\tau$-convergent to  $P(x)$ for each $x\in X$ if and only the following two conditions hold for each $x\in X$:
\begin{enumerate}
\item $\tau-\lim\frac{ T^n(x)}{n}=0$ and
\item $\{T_{[n]}(x)\colon n\in \N\}$ is relatively weakly compact.
\end{enumerate}
\end{theorem}
\begin{proof}
The necessity of condition (1) comes from
\[\frac{T^n(x)}{n}=T_{[n]}(x)-\frac{n-1}{n}T_{[n-1]}(x), \quad \mbox{for any } x\in X.\]
Since $\tau$ is stronger than the weak topology, Condition (2) is obviously also necessary.

We now check the sufficiency of (1) and (2).
We  first observe that $T_{[n]}(I-T)=\frac{1}{n}(T-T^{n+1})$, hence $T_{[n]}(I-T)(x)$  is $\tau$-convergent to 0 for every $x\in X$. Since, by the Banach Stenihaus theorem,   $(\|T_{[n]}\|)_n$ is a bounded sequence, we get
 \begin{equation}\label{yosi-t}\tau-\lim_n T_{[n]}(v)=0\mbox{ for every } v\in  \overline{\range(I-T)}.\end{equation}

Let now  $x\in X$. As $ \{T_{[n]}(x)\colon n\in \N\}$ is relatively weakly compact, we get from  Eberlein's theorem an increasing  sequence $(n_k)$ of natural numbers and $y_{_{x}}\in X$ such that $(T_{[n_k]}(x))$ is weakly convergent to $y_{_{x}}$. Then, by the preceding paragraph, and using that $T T_{[n_k]}=T_{[n_k]}T$,
\begin{equation}\label{yosfixed}
0=\sigma(X,X^\ast)-\lim_k
T_{[n_k]}\left((I-T)(x)\right)=y_{_{x}}-T(y_{_{x}}).
\end{equation}
On the other hand
\begin{align*}
x-y_{_{x}}&=\sigma(X,X^\ast)-\lim_k \frac{1}{n_k} \left(\sum_{n=1}^{n_k}(x-T^n(x))\right)\\&=\sigma(X,X^\ast)-\lim_k \frac{1}{n_k}\Bigl(\sum_{n=1}^{n_k}(I-T)
  \left((I+T+\cdots+T^{n-1})(x)\right)\Bigr).
\end{align*}
From this and  $\range(I-T)$ being a vector subspace we deduce that $x-y_{_{x}}$ is in  the norm closure $\overline{\range{I-T}}$. Applying \eqref{yosi-t} once more we see that $\tau-\lim_n T_{[n]}(x-y_{_{x}})=0$. From this and \eqref{yosfixed} we conclude  that, putting $Px=y_{_{x}}$,
\[\lim_n T_{[n]}x=Px, \quad \mbox{ for every } x\in X. \]
Since $\norm{y_{_{x}}}\leq \sup_k \norm{T_{[n_k]}}\norm{x}$, it is clear that $P\in\ll(X)$.
\end{proof}
\begin{corollary}
\label{MEth} (Mean ergodic theorem) If $X$ is reflexive, the sequence $(T_{[n]})_n$ is bounded (in other words, $T$ is Ces\`aro bounded) and ${\displaystyle \lim_n \frac{ T^n}{n}=0}$  in the strong operator topology, then $T$ is mean ergodic.
\end{corollary}
\begin{proposition}
\label{cbradio}
If $T$ is Ces\`aro bounded then $r(T)\leq 1$.
\end{proposition}
\begin{proof}
 If $r(T)>1$ then there is $M>1$ such that,  eventually, $\norm{T^n}\geq M^n$, which makes the sequence $(\frac{T^n}{n})_n$ unbounded. Since  $\frac{T^n}{n}=T_{[n]}-\frac{n-1}{n}T_{[n-1]}$, this goes against Ces\`aro-boundedness of $T$.
\end{proof}
\begin{corollary}
  Let $X$ be a Banach space.
     \label{lem:as}If $T$ is weakly mean ergodic, then $r(T)\leq 1$.
  %\end{enumerate}
\end{corollary}

 And, next, we collect some basic facts on uniform mean ergodicity.
 (1) is an immediate consequence of the spectral radius formula, (2) comes again from $T_{[n]}-\frac{n-1}{n}T_{[n-1]}=\frac{T^n}{n}$ and (3) is the classical Yosida Kakutani theorem \cite[Theorem 4, Corollary (i)]{YK}.
 \begin{theorem}\label{norms}
 Let  $T$ be a bounded linear operator on  a Banach space $X$:
\begin{enumerate}
\item \label{r(T)}If $r(T)<1$ then $T$ is uniformly mean ergodic (further, $T^n$ converges to 0 in the norm operator topology).
%\item \label{T-PT}$T$ is uniformly mean ergodic if and only if $T-P_T$ is uniformly mean ergodic.
\item\label{UMEnorm0} If $T $ is uniformly mean ergodic,  then ${\displaystyle \lim_{n\to \infty}\frac{\norm{T^n}}{n}=0}$.

\item \label{compactme} If $T$ is power bounded, compact and mean ergodic then it is uniformly mean ergodic.
\end{enumerate}

\end{theorem}

Uniform mean ergodicity can  be characterized  in a functional analytic way.
  \begin{theorem}[Dunford, Theorem 8 of \cite{Dunford43}, Lin, Theorem of \cite{lin74}]\label{lin74t}

Let $T$ be a bounded linear operator on the Banach space $X$.  The following assertions are equivalent:
\begin{enumerate}
  \item $T$ is uniformly mean ergodic.
  \item $\range(I-T)^2$ is closed and $\lim_n\frac{\|T^n\|}{n}=0$.
  \item Either $1\in \varrho(T)$ or $1$ is a pole of order 1 of the resolvent mapping $R(z,T)$ and $\lim_n\frac{\|T^n\|}{n}=0$.
  \item $\range(I-T)$ is closed and $\lim_n\frac{\|T^n\|}{n}=0$.
 \item $\range(I-T)$ is closed, $X=\range(I-T)\oplus \ker(I-T)$ and $\lim_n\frac{\|T^n\|}{n}=0$.
\end{enumerate}

If in addition we have $\ker (I-T)=\{0\}$ then all the satements are equivalent to
\begin{itemize}
\item[{\em (6)}] $1\notin \sigma(T)$ and $\lim_n\frac{\|T^n\|}{n}=0$.
\end{itemize}
\end{theorem}

The equivalence of the first 3 conditions and that they imply (4) and (5) were proved by Dunford  \cite[Theorem 8]{Dunford43}. Dunford also proved that (4) implies (1)  if  $T$ is assumed to be  mean ergodic.
 Lin proved in \cite{lin74} that (4) implies (1) with no extra assumptions, and then also (5), which certainly implies (4). Condition (6) is simply a particular case of the theorem which is specially relevant in our work. The decomposition in direct sum of  (5) when there are not fixed points is equivalent to  $I-T$ being an injective operator and $\range(I-T)=X$, i.e. to $I-T$ being an isomorphism.

We finish this subsection stating a classic result of Lotz, in the form that we need it

\begin{theorem}[Theorem 5 of \cite{lotz85}]
\label{Lotz85}
Let $(\Omega,\Sigma,\mu)$ be a positive measure space. Then every bounded linear operator $T$ on $L^\infty(\Omega)$  satisfying $\lim_n\norm{T^n}/n=0$ which is mean ergodic is also uniformly mean ergodic.

\end{theorem}

\subsection{Convolution operators}
Our convolution operators will be defined on Banach  spaces of functions on a locally compact group. For any such $G$, we will the denote by $\mg$ its Haar measure, i.e., its (essentially unique)  left invariant   measure. The  Banach space of  (equivalence classes of) Haar $p$-integrable functions will be simply  denoted as  $L^p(G)$.

The measure $\mg$ needs not be right invariant. Its behaviour   under  right transaltions  is gauged by  the modular function $\Delta_{\G }$. $\Delta_{\G}$ is a homomorphism of $G$ into the multiplicative group of positive real numbers such that
    \[\int f(xy^{-1})\dmg(x)=\Delta_{\G}(y)\int f(x)\dmg(x).\]

The    Banach space of bounded regular measures on $G$ will be denoted by $M(G)$. Hence   $M(G)=C_0(G)^\ast$. The weak* topology $\sigma(M(G),C_0(G))$ will  be referred to as the \emph{vague topology}.
 %We will at some points deal with  probability measures %on $G$, that is positive measures in the unit sphere of %$M(G)$.

We will always regard $L^1(G)$ as a subalgebra (an ideal, actually) of $M(G)$ through the embedding $f\mapsto f\cdot \mg$.
\begin{definitions}
\label{defs:co}
Let $\mu, \mu_1, \mu_2 \in M(G)$ be  bounded regular measures. We consider:
\begin{enumerate}
\item The \emph{convolution of measures}:
\[\langle \mu_1\ast\mu_2, f\rangle=\int\int f(xy)\dmu_1(x)\dmu_2(y), \mbox{ for every $f\in C_{00}(G)$}.\]
\item  The left \emph{convolution operator $1\leq p\leq \infty$}:
\[\lambda_p(\mu)\colon L^p(G)\to L^p(G) \]
  given by\[\lambda_p(\mu)(f)(s) =(\mu\ast f)(s):=\int f(x^{-1}s)\dmu(x), \quad \mbox{  $f\in L^p(G)$,  $s\in G$}.\]

\item The right \emph{convolution operator for $1\leq p\leq \infty$}:
\[\rho_p(\mu)\colon L^p(G)\to L^p(G) \]
  given by\[\rho_p(\mu)(f)(s)  =\int \Delta_{\G}(x)f(sx)\dmu(x), \quad \mbox{$f\in L^p(G)$,  $s\in G$}.\]
\end{enumerate}
\end{definitions}
The ergodic behaviour of the operators  $ \lambda_p(\mu) $ and  $ \rho_p(\mu) $ is the same when $1<p<\infty$. This is due to the following easily verifiable fact.
\begin{fact}
\label{intert}
 Let $1<p<\infty$ The operators  $\rho_p (\mu) $ and $\lambda_p (\mu)$ are intertwined by the linear isometry $U_p\colon L^p(G)\to L^p(G)$, given by $U_p(f)(s)=\Delta_{\G}^{1/p}(s^{-1})f(s^{-1})$, i.e.
  $\lambda_p(\mu)  U_p=U_p\rho_p(\mu)$.\end{fact}

  Proofs for items (1) and (2) of the following theorem can be found in Theorem 20.12 of \cite{hewiross1}, for  a proof of item (3), see  e.g., page 47 of \cite{green69}.
\begin{theorem}\label{young}
Let $G$ be a locally compact group and let $\mu \in M(G)$. Then:
\begin{enumerate}
\item
     $\mu \ast f\in L^p(G)$ for every $f\in L^p(G)$ and ${\displaystyle \|\mu \ast f\|^p\leq \|\mu\| \cdot \|f\|^p}$. As a consequence,
      \item $\lambda_p(\mu)\in \ll(L^p(G))$ and $\|\lambda_p(\mu)\|\leq \|\mu\|$.
\item $\norm{\lambda_1(\mu)}=\norm{\mu}$.
\end{enumerate}
\end{theorem}

$M(G)$  is a $\ast$-Banach algebra under convolution with involution defined by $\<{\mu^\ast,h}=\overline{\<{\mu,\tilde{h}}}$ for every $h\in C_{0}(G)$, where $\tilde{h}(s)=\overline{h(s^{-1})}$, for every $s\in G$.
    If $\mu=\mu^\ast$, the measure is called \emph{Hermitian}. Since $\lambda_2(\mu)^\ast=\lambda_2(\mu^\ast)$, the operator $\lambda_2(\mu)$ is self-adjoint precisely when $\mu$ is Hermitian. We will say that the measure $\mu$ is \emph{operator-normal} when  $\mu\ast\mu^\ast=\mu^\ast\ast \mu$.
The operator $\lambda_2(\mu)$ will be normal precisely when $\mu$ is operator-normal, in particular when $G$ is commutative, \cite[Theorem 20.23]{hewiross1}.

%
%The following is well-known and will be useful in the sequel
%\begin{lemma}\label{l1meas}
%Let $G$ be a locally compact group and $\mu\in M(G)$.
%\begin{enumerate}
%  \item (see, e.g.,
%  \cite[Page 47]{green69}) $\norm{\lambda_1(\mu)}_{\oplu }=\norm{\mu}$.
%      \item  $\sigma(\lambda_2(\mu))\subseteq \sigma(\lambda_p(\mu))\subseteq \sigma(\lambda_1(\mu))$.\footnote{\textcolor{blue}{No ?`Mentira???}}
%\end{enumerate}
%\end{lemma}

We will often deal with $\langle L^p(G),L^{p^\prime}(G)\rangle$ dualities. For $p\neq 2$ and 1/p+1/$p^\prime=1$ , this duality is given by:
\[ \langle f, g\rangle =\int f(x)g(x)\, \dmg(x)\quad \mbox{  $f\in L^p(G)$,  $g\in L^{p^\prime}(G)$}.\]
When $p=2$, the duality is derived  from the inner product in $L^2(G)$ and is therefore given by
\[ \langle f, g\rangle =\int f(x)\overline{g(x)}\, \dmg(x)
\quad \mbox{} f,\, g\in L^2(G).\]
 \subsection{Abelian groups: Fourier-Stieltjes transforms}\label{sec:2.3}
If $G$ is Abelian, the Fourier-Stieltjes transform establishes a unitary equivalence between convolution operators and multiplication operators. We use this subsection to recall this fact and some of its consequences.

Let $G$ be a locally compact Abelian and let $\T$ denote   the multiplicative group of complex numbers of modulus 1. By a \emph{character} of $G$ we understand a continuous homomorphism of $G$ into $\T$. The set $\widehat{G}$ of all continuous characters of $G$  with the topology of uniform convergence on compact sets  acquires the   structure of a locally compact Abelian group under pointwise multiplication.

 If $\mu\in M(G)$, the function $\widehat{\mu}\colon \widehat{G}\to \C$ given by $ \widehat{\mu} (\chi)=\int \chi(t)\,d\mu(t)$ is  a bounded uniformly continuous function known as  the \emph{Fourier-Stieltjes transform } of $\mu$.
The Riemann-Lebesgue theorem shows that, $\widehat{f}\in C_0(\widehat{G})$   for every $f\in L^1(G)$ (recall that we  simply write $f$ for the measure $f\cdot \dmg$).
The Fourier-Stieltjes transform  restricted to $L^1(G)$ is usually  known as the \emph{Fourier transform}.

  The symbol $M_{\widehat{\mu}}$ will denote the multiplication operator $M_{\widehat{\mu}}(f)\colon L^2(\widehat{G})\to L^2(\widehat{G})$  given by $M_{\widehat{\mu}}(f)(\chi)= \widehat{\mu}(\chi)\cdot f(\chi)$.
%\[\widehat{\lambda_2(\mu) f}=\widehat{(f\ast\mu)}=\widehat{f}\cdot \widehat{\mu}=M_{\widehat{\mu}}\widehat{f}.\]{\tt He de buscar algunas refs para esto.}
\begin{theorem}\label{FSUE}
 Let $G$ be a locally compact Abelian group and let  $\mu\in M(G)$.
 \begin{enumerate}
 \item The convolution operator $\lambda_2(\mu)\in \ll(L^2(G))$  is unitarily equivalent to the multiplication operator $M_{\widehat{\mu}}\in  \ll(L^2(\widehat{G}))$. Hence,  $\|\lambda_2(\mu)\|=\|\widehat{\mu}\|_\infty$.
%\item The operators $\widehat{\lambda_2(\mu)_{[n]}}$ and  $\mathcal{M}_{\mu,n}$ are unitarily equivalent.
    \item Let $K_\mu:=\{\widehat{f}\colon f\in \ker(I-\lambda_2(\mu))\}$ and let $\widehat{P_\lambda(\mu)}\in \ll(L^2(\widehat{G}))$ be  the projection operator onto  $K_\mu$. Then $\widehat{P_\lambda(\mu)}=M_{\Cf_{A_\mu}}$, where $A_\mu=\widehat{\mu}^{-1}(\{1\})$.
 \end{enumerate}
\end{theorem}

\subsection{Amenable groups. Spectrum of convolution operators}
 Ame\-na\-bility is a far-reaching  property  of topological groups  that is characterized by  the existence of invariant means on their space of uniformly continuous  bounded functions. All compact groups and all locally compact Abelian groups are amenable. Free groups and locally compact groups containing them, such as semisimple Lie groups, are among the most prominent nonamenable groups.   We will refrain   from defining what an amenable group is. For our purposes, we only need to recall the following  properties.
\begin{theorem}\label{cornormamen}
  Let $G$ be  a locally compact group, let  $\mu$ be a positive measure on $G$ and let $1\leq p<\infty$. Then:
  \begin{enumerate}
     \item (Theorem 3.2.2 of \cite{green69}, for instance)  If $G$ is amenable, then $\norm{\mu} \in \sigma(\lambda_p(\mu))$.
     \item (\cite[Th\'eor\`eme 5]{bergchrist74}) $H_\mu$ is amenable if and only if $r(\lambda_p(\mu))=\norm{\mu}$.
   \end{enumerate}
\end{theorem}

Proposition \ref{lemma6.6}  below is well known when $G$ is Abelian. It is also well-known   that $\sigma(\lambda_p(\mu))\subseteq \sigma(\lambda_1(\mu)) =\sigma(\mu)$ for every locally compact group, not necessarily amenable, and every $\mu\in M(G)$. The inclusion $\sigma(\lambda_2(\mu))\subseteq \sigma(\lambda_p(\mu))$ for amenable groups is not,  as far as we know,  explicitly stated in the literature.
 \begin{proposition} \label{lemma6.6}
 Let $G$ be an  amenable  locally compact and let $\mu\in M(G)$. Then, for any $1\leq p \leq q\leq 2$ or $2\leq q\leq p$,   $\sigma(\lambda_q(\mu))\subseteq \sigma(\lambda_p(\mu))$.
 \end{proposition}
 \begin{proof}
   Let $z\in     \C$, $z\notin \sigma(\lambda_p(\mu))$. Then there is an operator $S\in \ll(L^p(G))$, such that $(\lambda_p(\mu)-zI)\circ S=S\circ (\lambda_p(\mu)-zI)=I$.  Since $(\lambda_p(\mu)-zId)$ commutes with right translations, so will do its inverse $S$. In the terminology of \cite{herz71} this means that $S\in \mathrm{Conv}
   _p(G)$. By \cite[Theorem C]{herz71} (proved in Corollary of \cite{herzrivi72}, see also Section 8.3 of \cite{deri11}) we have then that $S\in \ll(L^q(G))$, what means that $z\notin \sigma (\lambda_q(G))$.
 \end{proof}
\subsection{Fixed points of convolution operators}\label{subs:fixed}
%From a Dynamics point of view,   fixed points   play an important r\^ole.
 The following theorem is part of  Corollary 6.6 of \cite{neufsalmiskalspro}. For $\mu\geq 0$, it can be deduced  from the results  of  Derriennic \cite[Th\'eor\`eme 8]{derr76}   (also obtained   by Mukherjea \cite[Theorem 2]{mukh76} in the second countable case) and D\'erriennic and Lin, \cite[Proposition 2.1]{derrlin89}.

\begin{theorem}
\label{nocompact2}
Let $G$ be a locally compact group and let $\mu\in M(G)$ be a  measure with $\norm{\mu}\leq 1$  such that $H_\mu$ is not compact.
If $f\in L_p(G)$, $1\leq p<\infty$, and $\mu*f=f$ (a.e.),   then $f=0$ (a.e.)
\end{theorem}

%\begin{proof}
%If $f\in C_0(G)$ satisfies $\mu*f=f$ (a.e.) then for all $x\in G$ such that  $\mu*f(x)=f(x)$  we have, by Proposition \ref{mun}, that
%$f(x)=(\mu^n\ast f)(x)$. Since $(\mu^n\ast f)(x)=\<{\mu^n,\widecheck{f_x}}$ converges to 0, it follows that $f(x)=0$.
%
%If $f\in L^p(G)$ and $\mu*f=f$  then, for each $\varepsilon>0$,  we can use an approximate identity  to find $g_{\varepsilon}\in C_{00}(G)$ such that $\|f-f\ast g_{\varepsilon}\|<\varepsilon$. Since $f\ast g_{\varepsilon}\in C_0(G)$ satisfies $\mu*f*g_{\varepsilon}=f*g_{\varepsilon}$, we have, by the preceding paragraph that $f\ast g_{\varepsilon}=0$ and, hence, that
%\[\|f\|_p\leq \|f-f\ast g_{\varepsilon}\|_p\leq \varepsilon,\]
%Thus,  $\|f\|_p=0$.
%\end{proof}
 %For results (probably the only ones in  the literature) on the existence of fixed points of $\lambda_p(\mu) $ for nonpositive measures the reader can check  Section 7 of \cite{neufsalmiskalspro}.

%
%  What makes this case worth separate study is that  these operators have no nontrivial fixed points, a fact which is always  a relevant from the Dynamics point of view.
  The impact of Theorem \ref{nocompact2} in the  ergodic behaviour of convolution operators is reflected in the following two consequences

 \begin{proposition}
 \label{nocompactme}
 Let $G$ be a locally compact group and let $\mu\in M(G)$  be a measure with $\norm{\mu}\leq 1$ such  that  $H_\mu$ is not compact. Then $(\lambda_p(\mu_{[n]})f)_n$ converges to 0 for each $f\in L_p(\mu)$.
 \end{proposition}
  \begin{proof}
  Since $L_p(\mu)$ is reflexive and $\|\lambda_p(\mu)\|\leq 1$, we get from Corollary \ref{MEth} that $(\lambda_p(\mu))$ is mean ergodic, and $(\lambda_p(\mu_{[n]}))_n$ converges in the strong operator topology to the projection $P$ on the fixed points of $\lambda_p(\mu)$. Theorem \ref{nocompact2} yields $P=0$.

  \end{proof}

\begin{proposition}\label{con+lin:mu}
Let $G$ be a locally compact group and let $\mu\in M(G)$  be a measure with $\norm{\mu}\leq 1$ such  that  $H_\mu$ is not compact. The following assertions are equivalent for $1\leq p<\infty$:
\begin{enumerate}
  \item $\lambda_p(\mu)$ is uniformly mean ergodic.
  %\item $1$ is not an approximate eigenvalue of $\lambda_p(\mu)$.
  \item $1\notin \sigma(\lambda_p(\mu))$.
  %\item There is a constant $C>0$ such that
%  \[\norm{\mu\ast f-f}_p\geq C\norm{f}_p, \quad \mbox{ for all $f\in L^p(G)$}.\]
\end{enumerate}
\end{proposition}\begin{proof}
Follows from Theorem \ref{nocompact2} and (6) of Theorem \ref{lin74t}.
\end{proof}

\subsection{Vague ergodicity}
Using convolution as multiplication, it also makes sense to consider the ergodic behaviour of a measure without reference to the operator it induces. The limitting process is then studied under the   $\sigma(M(G),C_0(G))$ topology which is usually called the \emph{vague} topology. Notice that by the Alaoglu-Bourbaki theorem,  this topology agrees, on bounded subsets of $M(G)$,  with the $\sigma(M(G),C_{00}(G))$-topology. Here $C_{00}(G)$ stands for the space of continuous functions on $G$ with compact support.

\begin{definition}
  Let $G$ be a locally compact group and let $\mu \in M(G)$. For $n\in \N$, we define $\mu_{[n]}=\frac{\mu+\mu^2+\cdots+\mu^{n}}{n}$. We say that $\mu$ is a \emph{vague-ergodic measure} if there is a measure $\mu_c\in M(G)$ such  that, in the vague topology,
  \[\lim_{n \to \infty} \mu_{[n]}=\mu_c.\]
 \end{definition}
 %\begin{remark}
% \label{meb}
% If $\mu$ is an ergodic measure then $(\|\mu^{n}\| / n)$ is a bounded sequence. That is a consequence of the  convergence of $\mu_{[n]}-\frac{n-1}{n}\mu_{[n-1]}$  to 0 in the vague topology together with the Banach Steinhauss principle.
% \end{remark}

Probability measures are always vague-ergodic. We recall here  this important classical result that can be deduced from the Mean Ergodic Theorem. See Theorem \ref{grenander} below, for the proof of a more  general version.
\begin{theorem}[Theorem 3.0 of \cite{grena68}]\label{idem}
 If $G$ is a second  countable  locally compact group, then every probability  measure $\mu$ is vague-ergodic   and if $\mu_c=\lim_{n \to \infty} \mu_{[n]}$, then $\mu_c$ is  a convolution idempotent measure.
\end{theorem}

\section{General  results}
In this section we develop tools that are not directly related with ergodicity but have a strong impact in our work.  Some hints on that impact are also included in this section.

We first  establish the continuity properties of the regular representation on $M(G)$. These properties,  will permeate most sections of the paper.

 Our second tool will be  a well-known result on the vague convergence of powers of probability measures. This has a clear  impact on the  existence of fixed points for the corresponding convolution operator.

% Our next concern is to extend  a result of Bekka on the spectra of  convolution operators from Hilbert spaces to the $L^p$-setting.  By Theorem \ref{lin74t}, these results have consequences on uniform mean ergodicity of the convolution operator.
%  We close the section proving that to determine the ergodic properties of $\lambda_p(\mu)$, $\mu\in M(G)$,  we can assume that the  subgroup generated by the support of $\mu$ is dense in $G$. Such measures are usually known as \emph{adapted measures}.

\subsection{Continuity properties of regular representations on $M(G)$}
We next establish the continuity properties of the left regular representation on $M(G)$. This is necessary  to connect vague-ergodicity of a measure with the mean ergodicity of the corresponding convolution operator.
\begin{proposition}
\label{rem1}
Let $G$ be a  locally compact group and consider the  mapping $\lambda_p: M(G)\to \ll(L^p(G))$.
\begin{enumerate}
\item  $\lambda_p$ is vague-WOT   continuous  on norm bounded subsets of $M(G)$, for every $1<p<\infty$.
\item   If $H_\mu$ is compact, then  $\lambda_p$ is vague-SOT sequentially continuous for every $1\leq p<\infty$.
\end{enumerate} \end{proposition}
\begin{proof}
We start with (1).
Let $(\mu_\alpha)_\alpha$ be a (norm) bounded net that  converges vaguely to $0$ . Since $(\lambda_p(\mu_\alpha))_\alpha$ is   also norm bounded, and  $C_{00}(G)$ is norm dense in $L^p(G)$, we only need to show that  $(\lambda^{p}(\mu_{\alpha})f)_\alpha$ is weakly convergent to  0 in $L^p(G)$ for each $f\in C_{00}(G)$.  As  $(\lambda^{p}(\mu_{\alpha})f)_\alpha$ is a bounded net in $L^p(G)$, and then relatively weakly compact, and  $\sigma(L^p(G),C_{00}(G))$ is a Hausdorff topology, it will again  suffice to prove that $\lim_\alpha \<{\lambda_p(\mu_\alpha)f,h}=0$ for every $h \in C_{00}(G)$.

%\noindent\textbf{Case I: non-compact $G$.}

   So,  let $f,h\in C_{00}(G)$.
   We first observe that, by Fubini's theorem,
   \[ \<{\lambda_p(\mu_\alpha)f,h}=
   \<{\mu_\alpha,\left(h\ast\check{f}}\right),\]
   where the first bracket refers to the $(L^p(G),L^{p^\prime}(G))$, the second to the $(M(G),C_{00}(G))$-duality and, for any $u\colon G\to \C$, $\check{u}(t)=u(t^{-1})$.

   Once this is clear, it immediately follows from the vague convergence of $(\mu_\alpha)_\alpha$ and the inclusion $C_{00}(G)\ast L^p(G)\subseteq C_0(G)$,
   that
   \[\lim_\alpha \<{\lambda_p(\mu_\alpha)(f),h}
 =\lim_\alpha \<{\mu_\alpha,\left(h\ast\check{f}}\right)=0.\]

   For Statement (2) we start with a \emph{sequence} $(\mu_n)_n$ vaguely convergent to 0.
    %it will be again su
%  Since bounded sequences in $L^p(\mu)$ that are pointwise convergent are always weakly convergent (see, e.g., \cite[Theorem 13.44]{hewistrom}), we only have to show that $(\lambda^{p}(\mu_{n})(f)(x))_n$ converges to 0 for every $x\in G$ and every $f\in C_{00}(G)$.
%
%  We fix $f\in C_{00}(G)$ and $x\in G$. Then
%  \begin{align*}
%   \lambda_p(\mu_n)(f)(x)&=\left|\int f(t^{-1}x)\,d\mu_n(t)\right|
%   =\langle\check{f}_x,\mu_n\rangle
%    \end{align*}
%where , $\check{f}_x(t)=f(t^{-1}x)$ for every $t\in G$,  and the bracket refers to the $(C_{00}(G),M(G))$-duality. Since $\mu_n$-converges to 0 in the $\sigma(M(G),C_{00}(G))$-topology, the convergence of $\lambda_p(\mu_n)(f)(x)$ follows.
%
%\noindent\textbf{Case II, compact $G$:} Note that in this case $\dmg$ is a finite measure.
As before, we only need to check that $\norm{\lambda_p(\mu_n)f}_p$ converges to 0 for every $f\in C_{00}(G)$. Let $K:=\supp(f)$

% for each  $n\in \N$,
% %$\psi_n(x)=\left|\int \psi_x(t)d\mu_n(t)\right|$ or, what is the same,
%  $\psi_n=   \left|\mu_n \ast f
%  \right|$.%
%With the notation of the previous case, we have that
Now,
  \begin{align*}
  \bnorm{\lambda_p(\mu_n)f}_p^p= \int \left|\<{\mu_n,\widecheck{f_x}}\right|^p \,dx,
%\left|\bang{\lambda_1(\mu_n)f,h}\right|&\leq \int \psi_n(x)\,dx.
 \end{align*} where
 $f_x$ denotes the right translate of $f$ by $x$, so that $\widecheck{f_x}(t)=f(t^{-1}x)$.
% \begin{align*}
%\left|\bang{\lambda_1(\mu_n)f,h}\right|&\leq \int \psi_n(x)\,dx.
% \end{align*}
As above, vague convergence of $(\mu)_n$ implies that    $\lim_n \left|\<{\mu_n, \widecheck{f_x}}\right|^p=0$
for every $x\in G$. Taking into account
\[\left|\<{\widecheck{f_x},\mu_n}\right|^p\leq  \norm{f}_\infty^p \cdot \sup_n \norm{\mu_n}^p_{_{M(G)}}\Cf_{H_\mu K}(x),\]
we can apply Lebesgue dominated convergence to conclude   $\lim_n \norm{\lambda_p(\mu_n)f}_p=0$, as we wanted to prove.
\end{proof}
Vague-WOT continuity of $\lambda_1$ and Vague-SOT convergence of $\lambda_p$ ($1<p<\infty$)  on bounded sets, both fail when $G$ is not compact, as the following example shows.
\begin{example}\label{exNONcont}
If $G$ is an infinite discrete  group,   the mapping $\lambda_1\colon M(G)\to \ll(\ell^1(G))$  is not  vague-WOT continuous on bounded sets,  and, for $1<p<\infty$,  $\lambda_p\colon M(G)\to \ll(\ell^p(G))$ is not vague-SOT sequentially continuous.
\end{example}
\begin{proof}
We prove that $\lambda_p$  fails to be vague-SOT sequentially continuous for $1\leq p<\infty$,. Since $\ell^1(G)$ has the Schur property, this implies both statements. Let $(s_n)_n$ be  a sequence in $G$ with infinitely many different terms. The sequence  $(\delta_{s_n})_n $ clearly converges vaguely to 0. Take any  $f\in \ell^1(G)$,  $f\neq 0$. Then, for any $1\leq  p<\infty$,
  \[ \norm{\lambda_p(\delta_{s_n})f}_p=\norm{f}_p.\]
  Hence, $(\lambda_p(\delta_{s_n})f)_n$ cannot converge to 0 in $\ell^p(G)$.
\end{proof}
Theorem \ref{rem1} establishes an immediate relation between vague-ergodicity and mean ergodicity of convolution operators.

\begin{proposition}\label{allp}
Let $G$ be a  locally compact group and $\mu\in M(G)$. The following statements are equivalent:
\begin{enumerate}
\item $\mu$ is vague-ergodic.
%\item $(\mu_{[n]})_n$ is norm bounded.
\item $\mu$ is Ces\`aro bounded and $\lim_n\frac{\mu^{n}}{n}=0$, vaguely.
     \item $\mu$ is Ces\`aro bounded  and $\lambda_p(\mu)$ is weakly mean ergodic for all $1<p<\infty$.
 \item $\mu$ is Ces\`aro bounded  and $\lambda_p(\mu)$ is weakly mean ergodic for some $1<p<\infty$.
\end{enumerate}
If $H_\mu$ is compact we can replace (3) and (4) by
\begin{enumerate}
  \item[(3)$^\prime$] $\mu$ is Ces\`aro bounded and $\lambda_p(\mu)$ is  mean ergodic for all $1\leq p<\infty$.
  \item[(4)$^\prime$]   $\mu$ is Ces\`aro bounded and $\lambda_p(\mu)$ is  mean ergodic for some $1\leq p<\infty$.
\end{enumerate}
 \end{proposition}

\begin{proof}
We prove first the necessity of the conditions. If $\mu$ is vague-ergodic it is clear that $(\norm{\mu_{[n]}})_n$ is bounded. The equality,
\[\frac{\mu^{n}}{n}=\mu_{[n]}-\frac{n-1}{n} \mu_{[n-1]},\]
shows that  $\frac{ \mu^n}{n}$ converges vaguely to 0. Therefore (1) implies (2)

We assume now that  (2) holds and fix $1<p<\infty$.  We deduce from Proposition \ref{rem1}  that $(\lambda_p(\mu)_{[n]})_n$ is a bounded sequence in $\ll(L^p(G))$ and  that $\lim \lambda_p(\mu)^n/n=0$ in the weak operator topology. Theorem \ref{weakyosida} then implies that $\lambda_p(\mu)$ is {\em weakly mean ergodic}.

Statement (3) certainly implies (4).  Assume now (4), i.e., that   $(\mu_{[n]})_n$ is norm bounded and  $(\lambda_p(\mu)_{[n]})_n$ converges in the weak operator topology to some $P\in \ll (L^p(G))$. Since the sequence $(\mu_{[n]})_n$ is norm bounded, Theorem \ref{rem1} (i) implies that for any  accumulation point $\mu$  of $(\mu_{[n]})_n$, $\lambda_p(\mu)=P$.  It follows that $(\mu_{[n]})_n$ has only precisely one accumulation point, it is therefore convergent.

The proof of the case when $G$ is compact is completely analogous but using Proposition \ref{rem1} (2).
\end{proof}
Convolution operators $\lambda_p(\mu)$ can be mean ergodic and yet $\mu$ fail to be vague ergodic, even if $G$ is compact and Abelian, see Remarks  \ref{rks}. For an example with $\mu \geq 0$, see Example \ref{ex:free}.

The foregoing theorem and Corollary \ref{allp} lead to the following generalization of Theorem \ref{idem} that goes beyond probability measures and second countable groups.
\begin{theorem}
\label{grenander} Let $G$ be a  locally compact group and let $\mu\in M(G)$.
If $(\|\mu^n\|)_n$ is a bounded sequence, then $\mu$ is vague-ergodic.
\end{theorem}

As a further corollary we get the following version of Theorem \ref{nocompact2}.
% that
% an extension of  equivalence between (i) and (iii) in \cite[Theorem 7.1]{neufsalmiskalspro}, where the same statement is given for probablity measures.  There the authors show even that when $\mu$ is a probability measure and $H_\mu$ is not compact then $(\mu^n)_n$ converges vaguely to 0.

\begin{corollary}
\label{neufangextension}
Let $G$ be a  locally compact group and let $\mu\in M(G)$ with $\|\mu\|\leq 1$. Then   $(\mu_{[n]})_n$ is vague convergent to 0  if and only if $H_\mu$ is not compact.
\end{corollary}
\begin{proof}
By Proposition \ref{grenander} $\mu$ is vague ergodic, i.e. there is $\mu_c\in M(G)$ such that $(\mu_{[n]})_n$ is convergent to $\mu_c$ in the vague topology. Also by Proposition \ref{allp} $(\lambda_2(\mu_{[n]}))_n$ is convergent to $\lambda_2(\mu_c)$, which must be the projection on the fixed points of $\lambda_2(\mu_c)$. If $H_\mu$ is compact then the characteristic function $\Cf_{H_\mu}$ is a fixed point of $\lambda_2(\mu)$, and then $\lambda_2(\mu_c)$ is a non null projection, and hence $\mu_c\neq 0$. Conversely, if $H_\mu$ is not compact then $\lambda_2(\mu_c)=0$ by Theorem \ref{nocompact2}, and thus certainly also $\mu_c=0$.
\end{proof}

\begin{remark}
The condition $\norm{\mu}\leq 1$ in  Corollary \ref{neufangextension} is imposed by its dependence on Theorem \ref{nocompact2}. The simple proof of Corollary \ref{neufangextension} will remain valid under any condition  on $\mu$ that keeps $(\norm{\mu^n})_n$  bounded and makes sure that $\mu_c=0$.
Since $\mu_c$ is an idempotent measure, this latter condition is freely obtained in groups that do not admit  nontrivial idempotent measures.
Any locally compact Abelian group with no nontrivial compact subgroup satisfies this property, \cite[Theorem 3.3.2]{rudin}.
\end{remark}

\section{Operator-normal measures}\label{sec:ON}
In this section we restrict our study to  measures that give rise to convolution operators that are normal, i.e., to operator-normal measures according to our definitions in Section 2.   This will automatically involve all measures when $G$ is Abelian.

  Normal operators on Hilbert spaces satisfy the identity $r(T)=\norm{T}$, and that  greatly simplifies the  analysis of mean ergodicity. The  following is an easy consequence of Theorems \ref{basic:me} and  \ref{lin74t}.
\begin{theorem}
\label{normal}
Let $\h$ be a Hilbert space and let $T\in \ll(\h)$ be a normal operator. Then
\begin{enumerate}
  \item The operator $T$  is  mean ergodic if and only if it is weakly mean ergodic if and only if  $\|T\|\leq 1$.
  \item The  operator  $T$ is uniformly mean ergodic if and only if $\|T\|\leq 1$ and $1$ is not an accumulation point of $\sigma(T)$.
\end{enumerate}
\end{theorem}
\begin{proof}
Since $\|T\|=r(T)$, corollaries  \ref{MEth} and \ref{lem:as} imply that $T$  is  (weakly) mean ergodic if and only if $\norm{T}\leq 1$.

If $T$ is uniformly mean ergodic it follows from (1) implies (3) on Theorem \ref{lin74t} that 1 cannot be an accumulation point of $\sigma(T)$.

For the converse we only have to recall that for a normal operator $T$ and $\lambda\in \C$,  $\range(T-\lambda I)$ is closed if and only if $\lambda$ is not an accumulation point of $\sigma(T)$, see, e.g., \cite[Proposition 4.5 of Chapter XI]{conway}.  Hence $T$ is uniformly mean ergodic if  $1$ is not an accumulation point of $\sigma(T)$, by  (4) implies  (1)  of Theorem \ref{lin74t}.
\end{proof}

\begin{corollary}
\label{cuadrados}
Let $\h$ be a Hilbert space and let $T\in \ll(\h)$ be a normal operator. Then $T$ is uniformly mean ergodic whenever $T^2$ is.
\end{corollary}

\begin{proof}
If $T$ is normal also $T^2$ is. Assume that $T$ is not uniformly mean ergodic. Then  $1$ is an accumulation point of $\sigma(T)$, and therefore also is an accumulation point of
$\sigma(T^2)=\sigma(T)^2$. Therefore $T^2$ is not uniformly mean ergodic.
\end{proof}

We will see in Example \ref{ex:radnonp} that the converse of Corollary \ref{cuadrados} is not true.
\subsection{Mean ergodicity of normal convolution operators}

%Let $\mathbb{D}$ denote the open unit  disk in the complex plane.
We can  now completely characterize the mean ergodicity of $\lambda_2(\mu)$ when $\mu$ is operator-normal. This provides  a complete characterization of the mean ergodicity of $\lambda_2(\mu)$ when $G$ es abelian and $\mu\in M(G)$.

\begin{theorem}

\label{hilbert}
Let $G$ be a locally compact group $G$ and let  $\mu$ be an operator-normal measure on $G$. Then the following assertions are equivalent:
\begin{enumerate}
\item The operator $\lambda_2(\mu)$ is mean ergodic.
\item The operator $\lambda_2(\mu)$ is weakly mean ergodic.    \item $r(\lambda_2(\mu))=\norm{\lambda_2(\mu)}\leq 1$.
       \item The operator $\lambda_2(\mu)$ is power bounded.
\item The operator $\lambda_2(\mu)$ is Ces\`aro bounded.
           \end{enumerate}
           \end{theorem}
           \begin{proof}
             The equivalence of Statements (1) to (3) follows at once from Theorem \ref{normal}.   It is obvious that (3) implies (4) and (4) implies (5).  Proposition \ref{cbradio} and the normality of $\lambda_2(\mu)$  show that  Statement (5) implies Statement (3).\end{proof}
           \begin{remarks}\label{rks}
          The assumptions of    Theorem \ref{hilbert} cannot completely removed and its conclusions cannot be easily   strengthened.
           \begin{enumerate}
             \item Example 6.24 of \cite{schr70} shows that every nondiscrete      locally compact Abelian group contains a measure $\mu\in M(G)$ with $\norm{\mu^n}\geq 2^n$ and $r(\lambda_2(\mu))=\norm{\lambda_2(\mu)}=
             \norm{\widehat{\mu}}_\infty<1$.
             The operator $\lambda_2(\mu)$ is mean ergodic but the  measure $\mu$ is not Ces\`aro bounded, let alone vague-ergodic. After Theorem \ref{allp}, one deduces that, at least for operator-normal measures, vague-ergodicity is strictly stronger than weak mean ergodicity of the convolution operator.
\item Both the normality condition $\mu^\ast \ast \mu=\mu \ast \mu^\ast$ and the restriction to the Hilbert case $p=2$ can be removed if $\mu\geq 0$  and $H_\mu$ is amenable, see Theorem \ref{pb=me:amen} \emph{infra}. Example \ref{ex:free} shows that  Theorem \ref{hilbert} is no longer true  when  $\mu$ is not operator-normal and $H_\mu$ is not amenable, even if $\mu$ is positive.
    \item If we keep the condition $\mu^\ast \ast \mu=\mu \ast \mu^\ast$ but consider $p\neq 2$, the result also fails, as witnessed by the following example.
           \end{enumerate}
           \end{remarks}
        \begin{example}
        Let $p>2$. On $G=\Z_3$, the cyclic group of order 3, there is a measure $\mu\in M(G)$  such that  $\norm{\lambda_p(\mu)}>1$ but $\lambda_p(\mu)$ is uniformly mean ergodic.
        \end{example}\begin{proof}
         If $G=\{e,x,x^2\}$, let $\mu= (\delta_x-\delta_{x^2})$.  An elementary computation yields that, for any $p$. $r(\lambda_p(\mu))=\norm{\lambda_2(\mu)}
         =\sqrt{3}$.
         Computing $\norm{\lambda_p(\mu)f}_p$ with
         $f=6^{-1/p}\delta_e+
         6^{-1/p}\delta_x-(3/2)^{-1/p}\delta_{x^2}$, one sees that, for every $p$,
         \[\norm{\lambda_p(\mu)}^p\geq \frac{1}{3}\left(1+4^{\frac{1}{p}}\right)^p.\]
    It is easy to check then that for every $p>2$,
      $\norm{\lambda_p(\mu)}>\sqrt{3}$.
      Pick now $t$ with $\norm{\lambda_p(\mu)}>t>\sqrt{3}$. The measure $\mu_t=\frac{1}{t}\mu$ is then the desired measure with $\norm{\lambda_p(\mu_t)}>1$ but $r(\lambda_p(\mu_t))<1$.
              \end{proof}

%
%        \begin{example}
%        On $G=\Z_3$, the cyclic group of order 3, there are measure $\mu\in M(G)$  and $p>2$ such that  $\norm{\lambda_p(\mu)}>1$ but $\lambda_p(\mu)$ uniformly mean ergodic.
%        \end{example}\begin{proof}
%         If $G=\{e,x,x^2\}$, let $\mu_t=t (\delta_x-\delta_{x^2})$ with $t>0$.  An elementary computation yields that, for any $p$. $\norm{\lambda_2(\mu_t)}=r(\lambda_p(\mu_t))=\sqrt{3}t$.
%          The inequality \[\norm{\lambda_p(\mu_t)}\geq 2 \left(\frac{2}{3}\right)^{\frac{1}{p}} t\] follows right away from computing
%          $\lambda_p(\mu_t) f$  for $f= 3^{-1/p}\left(\delta_e-\delta_x+\delta_{x^{2}}\right)$.
%
%          If $p$ is large enough, any choice of $t$ with
%          \[\frac{1}{2}\left(\frac{3}{2}\right)^{\frac{1}{p}}<t<\frac{1}{\sqrt{3}}\]
%          provides the required example of an operator $\lambda_p(\mu)$ with $r(\lambda_p(\mu))<1$ and $\norm{\lambda_p(\mu)}>1$.
%        \end{proof}
\subsection{Uniform mean ergodicity of normal convolution operators}
Turning to uniform mean ergodicity the direct consequence of Theorem \ref{normal} is the following.
           \begin{theorem}
\label{hilbert:ume}
Let $G$ be a locally compact group $G$ and let  $\mu$ be an operator-normal measure on $G$. Then the following assertions are equivalent:
\begin{enumerate}
\item           The  operator $\lambda_2(\mu)$ is  uniformly mean ergodic.
  \item  $\norm{\lambda_2(\mu)}=r(\lambda_2(\mu))\leq 1$ and $1$ is not an accumulation point of  $\sigma(\lambda_2(\mu))$.
\end{enumerate}
\end{theorem}

We next extend the preceding Theorem to the case $p\neq 2$. This will follow from Proposition \ref{paracadap}    which ensures that if uniform mean ergodicity of $\lambda_q(\mu)$ is assumed for  some $q$, then    condition (3) in Theorem \ref{lin74t} can be relaxed as in Theorem \ref{normal}  to characterize uniform mean ergodicity for arbitrary $p$.
%
%The  results obtained so far  this Section  all rely on the relation $r(\lambda_2(\mu))=\norm{\lambda_2(\mu)}$  afforded by the normality of the operator $\lambda_2(\mu)$. In this sense they are exclusive of the case $p=2$. As a matter of fact, the behaviour of $\lambda_p(\mu)$ when $p\neq 2$ can be different, see for instance Example \ref{exl2diflp}.
%
% Results similar to Theorem \ref{hilbert} can be obtained for other $p$'s  as long as the  spectrum of $\lambda_p(\mu)$ can be controlled.
We first need a known result on extension of vector-valued holomorphic functions.

\begin{definition} Let $X$ be a Banach space. A subset $H\subseteq X^*$ is said to  be {\em separating} in $X$ when $x^*(x)=0$ for all $x^*\in H$ implies $x=0$.
\end{definition}

\begin{proposition}{\cite[Theorem 1]{GE2004},\cite[Corollary 10, Remark 11]{BFJ2007}}
\label{holomorphicextension}
Let $X$ be a Banach space, $H$  a separating subspace of $X^*$, $\Omega\subseteq \mathbb{C}$  a domain,  $a\in \Omega$. Let   $f :\Omega\setminus \{a\}:\to X$ be a holomorphic function such that $x^*\circ f$ admits holomorphic extension to $\Omega$ for each $x^*\in H$. Then $f$ admits a (unique) holomorphic extension to $\Omega$ .
\end{proposition}

\begin{proposition}
\label{paracadap}
Let $G$ be a locally compact group and let $\mu\in M(G)$. Assume there is $q>1$ such that $\lambda^q(\mu)$ is uniformly mean ergodic. Then, for each $1\leq p\leq \infty$, $\lambda_p(\mu)$ is uniformly mean ergodic if and only if  $\lim\|\lambda_p(\mu^{n})\|/n=0$ and $1$ is not an accumulation point of $\sigma(\lambda_p(\mu))$.
\end{proposition}

\begin{proof}
 The necessity follows from (1) implies (3) in Theorem \ref{lin74t}. We prove the converse. Let $r>0$ such that $(B(1,r)\setminus \{1\})\cap \sigma(\lambda_p(\mu))=\emptyset$ and $(B(1,r) \setminus \{1\})\cap \sigma(\lambda_q(\mu))=\emptyset$.  The  resolvent mapping restricted to $B(1,r)\setminus\{1\}$
\[R(\cdot,\lambda_p(\mu) )\colon B(1,r)\setminus \{1\} \to \ll(L^p(G))\]
  is then a holomorphic function. From uniform mean ergodicity in $L_q(G)$ and Theorem \ref{lin74t}   we get that, for each $f,g\in C_{00}(G)$, the function
\[z\mapsto \<{ (z-1)R(z,\lambda^q(\mu)(f)),g}, \]
 $z\in B(1,r)\setminus\{1\}$, is holomorphic and admits a holomorphic extension to 1. For $f,g\in C_{00}(G)$ define $I_{f,g}\in \ll(L_q(G))^*$ by $I_{f,g}(T)=\<{T(f),g}$.
 Observe that $\{I_{f,g}: \ f,g\in C_{00}(G)\}$ is a separating subset of $\ll(L_p(G))^*$. We conclude from $R(z,\lambda^q(\mu)(f)=R(z,\lambda_p(\mu)(f))$  for each $f\in C_{00}(G)$ that the function  $B(1,r)\setminus\{1\}\to \C$, $z\mapsto\<{ (z-1)R(z,\lambda_p(\mu))(f),g} $ admits a holomorphic extension to 1 for each $f,g\in C_{00}(G)$. The conclusion follows from Proposition \ref{holomorphicextension} and (3) implies (1)  in Theorem \ref{lin74t}.
\end{proof}

\begin{theorem}
\label{compactume}
Let $G$ be a locally compact group and let $\mu$ be an operator-normal  measure. Assume also that $H_\mu$ is amenable.
The following assertions are equivalent for $1\leq p\leq \infty$.
\begin{enumerate}
\item  $\lambda_p(\mu)$ is uniformly mean  ergodic.
\item  1 is not an accumulation point of $\sigma(\lambda_p(\mu))$, \emph{and}  $ \lim_n \norm{\lambda_p(\mu^{n})}/n=0$.
    \end{enumerate}

%For $p=1$, condition (2) can be replaced by $\lim_n\norm{\mu^{n}}/n=0$.
\end{theorem}
\begin{proof}
From $ \lim_n \norm{\lambda_p(\mu^{n})}/n=0$ and the spectral radius formula it follows that $r(\lambda_p(\mu))\leq 1$. Combining Theorem \ref{hilbert:ume}  with Proposition \ref{lemma6.6}   we get that  (2) implies that $\lambda_2(\mu)$ is uniformly mean ergodic.    Proposition \ref{paracadap} then  proves that Statement (2) implies Statement (1). The other direction comes directly from Theorem \ref{lin74t}.
\end{proof}

%\begin{remark}\label{rem:4.13}
  %If $\mu$ is operator-normal and  $p=2$, the condition  $ \lim_n\norm{\lambda_p(\mu^{n})}/n=0$ is equivalent to
    %$r(\lambda_p(\mu))\leq 1$. The  condition $\sigma(\lambda_2(\mu))\subseteq \lambda_p(\mu)$ holds in many situations, see Lemma \ref{spectro:p}. This shows that Theorem \ref{compactume} is a genuine extension of Theorem \ref{hilbert:ume}.
%\end{remark}
%
%\begin{corollary}
%Let $\mu$ be an operator  normal measure with natural spectrum, i.e $\sigma(\lambda_p(\mu))=\sigma(\lambda_2(\mu))$ for $1\leq p \leq \infty$. Then $\lambda_p$ is uniformly mean ergodic if and only if $ \lim_n \norm{\lambda_p(\mu^{n})}/n=0$ and $1$ is not an accumulation point of $\sigma(\lambda_2(\mu))$.
%\end{corollary}

\subsection{Abelian groups}
When $G$ is Abelian, Corollaries \ref{hilbert} and \ref{hilbert:ume} can be rephrased,  using Theorem \ref{FSUE},  in terms of Fourier-Stieltjes transforms.
\begin{corollary}
\label{pb=me+ume:ab}
Let $G$ be a locally compact Abelian group $G$ and let  $\mu\in M(G)$. Then:
\begin{enumerate}
  \item The operator $\lambda_2(\mu)$ is mean ergodic if and only if $\norm{\widehat{\mu}}_\infty\leq 1$.
         \item \label{5} The  operator $\lambda_2(\mu)$ is uniformly mean ergodic if and only if $\norm{\widehat{\mu}}_\infty\leq 1$ and  ${\displaystyle  1\notin \overline{\left\{\widehat{\mu}(\chi)\colon \chi \in \widehat{G}, \;\widehat{\mu}(\chi)\neq 1\right\}}}$.
\end{enumerate}
\end{corollary}
There is an obvious relation between mean ergodicity of $\lambda_2(\mu)$ and how $A_\mu$   embeds topologically in $\widehat{G}$. Recall, Section \ref{sec:2.3}, that $A_\mu=\widehat{\mu}^{-1}(\{1\})$ and hence that $A_\mu $ is always a closed set. We now clarify this.
\begin{corollary}\label{cor:ume:ab} Let $G$ be a locally compact abelian group and let $\mu\in M(G)$.  If $\lambda_2(\mu)$ is uniformly mean ergodic, then $A_\mu$ is a closed and open subset of $\widehat{G}$.
\end{corollary}
\begin{proof}
If $\lambda_2(\mu)$ is uniformly mean ergodic, we deduce from
\ref{5}  of Corollary \ref{pb=me+ume:ab} that  there is $\delta>0$ such that:
\[ \left|\widehat{\mu}(\chi)-1\right|>\delta \quad \mbox{ for all $\chi \in \widehat{G}$ with $\chi\notin A_\mu$}. \]
We see then that
$\hat{\mu}^{-1}(1-\delta,1+\delta)=A_\mu$ and therefore that $A_\mu$  is open.
\end{proof}

\begin{remark}\label{r4.13}It is easy to find measures with $A_\mu$ open  that do not produce uniformly mean ergodic operators $\lambda_2(\mu)$. Take for instance $G=\T$  and $\mu=\delta_{s}$ for any $s=e^{2\pi i x}$ with $x\notin \Q$. Identifying $\widehat{\T}$ with  the discrete group $\Z$  we have that  $\widehat{\mu}\colon \Z \to \C$ is defined by $\widehat{\mu}(k)=s^k$ for $k\in \Z$. The range of $\widehat{\mu}$ is well-known to be dense in $\T$, whence we see that $\mu$ does not satisfy condition \ref{5} of Corollary \ref{pb=me+ume:ab}.  $\widehat{\T}$ being discrete, $A_\mu$ is sure open.

The situation in this respect is quite different if  $\mu \in L^1(\T)$ or, more generally, if $\widehat{\mu}\in c_0(\Z)$, as we next see.
\end{remark}

\begin{definition}
  If $G$ is a locally compact Abelian group, we denote by $M_0(G)$ the set of measures $\mu\in M(G)$ such that $\widehat{\mu}\in C_0(\widehat{G})$.
\end{definition}

Since the Riemann-Lebesgue theorem proves that $L^1(G)\subset M_0(G)$, the following Corollary applies in particular to all $\lambda_2(f)$ with $f\in L^1(G)$.
\begin{corollary}
\label{ume:ab:l1}
  Let $G$ be an locally compact abelian group and let $\mu\in M_0(G)$. Then  $\lambda_2(\mu) $ is uniformly mean ergodic if and only if
  $\norm{\widehat{\mu}}_\infty \leq 1$ and $A_\mu$ is a closed and open subset of $\widehat{G}$.
   \end{corollary}
\begin{proof}
After Theorem \ref{pb=me+ume:ab} and Corollary \ref{cor:ume:ab}, we only have to prove the sufficiency part.

Suppose therefore that $A_\mu$ is an open (and closed) subset of $\widehat{G}$ but there is a sequence $(\chi_n)_n$ in $\widehat{G}$ with $ \chi_n\notin A_\mu$ for every $n\in \N$ such that $1=\lim_n \widehat{\mu}(\chi_n)$. Since $\widehat{\mu}\in C_0(\widehat{G})$, we deduce that there is a compact subset $K$ of $\widehat{G}$ such that  $\chi_n\in K$ for every $n$. There is then $\chi_0\in \widehat{G}$ and a subnet $(\chi_\alpha)_\alpha$ of $(\chi_n)_n$ such that $\lim_\alpha \chi_\alpha =\chi_0$. Since $ \lim_n \widehat{\mu}(\chi_n)=1$, we have that $\chi_0\in A_\mu$; but $A_\mu$ being  open, this would imply that the net $(\chi_\alpha)_\alpha$ is eventually in $A_\mu$.
\end{proof}
If $G$ is compact, then $\widehat{G}$ is discrete, we thus have:
\begin{corollary}
\label{ume:ab:compact}
  Let $G$ be a compact Abelian group and let $\mu\in M_0(G)$.  The following assertions are then equivalent:
\begin{enumerate}
\item $\lambda_2(\mu)$ is power bounded.
\item $\lambda_2(\mu)$ is mean ergodic.
\item $\|\widehat{\mu}\|_\infty\leq 1$.
\item  $\lambda_2(\mu)$ is uniformly mean ergodic.
   \end{enumerate}
\end{corollary}

Theorem \ref{compactume}  suggests that, for a given  $\mu\in M(G)$, the uniform mean ergodicity of $\lambda_p(\mu)$ may depend on $p$. To confirm this we  will  need the fllowing result due to S. Igari.
%, we then proceed to produce uniformly mean ergodic operators $\lambda_2(\mu)$ such that $\lambda_p(\mu)$  is not uniformly mean ergodic.
   \begin{theorem}[Particular case of Theorem 1 of \cite{iga69}]\label{igari}
  Let $G$ be a nondiscrete locally compact Abelian group, let $1\leq p<2$ and let $\Phi\colon [-1,1]\to \C$. If $\Phi$ does not extend to an entire function, there are  $\mu\in M(G)$ with $\widehat{\mu}(\widehat{G})\subseteq [-1,1]$, and $h\in L^p(G)$ such that  $(\Phi\circ \widehat{\mu} )\cdot \widehat{h}$ is not the Fourier transform of any function  in $L^p(G)$.
  \end{theorem}

  \begin{proposition}\label{exl2diflp}
  For any nondiscrete locally compact Abelian group $G$ and for every $1\leq p<2$ there is a measure $\mu\in M(G)$   such that  $\lambda_2(\mu)$ is uniformly mean ergodic but $\lambda_p(\mu)$ is not.
  \end{proposition}
  \begin{proof}
  Define $\Phi\colon [-1,1]\to \C$ by $\Phi(t)=1/(t-z)$ where $z\in \T\setminus[-1,1]$. Since $\Phi$ cannot be extended to an entire function we can find $\mu_z \in M(G)$ and $h_z\in L^p(G) $ with the properties of Theorem \ref{igari}.

  Suppose that  $z\notin \sigma(\lambda_p (\mu_z))$. In that case,  there  would be  an operator $T\in \ll{}(L^p(G))$ such that $\mu_z \ast Tf
- zTf =f$ for  every $f\in L^p(G)$. Taking Fourier-Stieltjes transforms we see that, for  every $f\in L^p(G)$,
\[ \widehat{Tf}\left(\widehat{\mu_z}-z\right)=\widehat{f}.\]
The preceding equality applied to $f=h_z$ implies then that
\[\widehat{Th_z}=( \Phi \circ \widehat{\mu_z})\cdot \widehat{h_z}, \]
which goes against the choice of $\mu_z$ and $h_z$ from Theorem \ref{igari}. We conclude that $z\in \sigma(\lambda_p (\mu_z))$.

Once we have found $\mu_z\in M(G)$ with $\widehat{\mu_z}(\widehat{G})\subseteq [-1,1]$ and $z\in \sigma(\lambda_p (\mu_z))$ we consider
$\mu=\overline{z}\mu_z$.  Then
\[\sigma(\lambda_2(\mu))=
\overline{z}\overline{\widehat{\mu_z}(\widehat{G})}\] is contained in a diameter of the unit circle not passing through 1.  According to Corollary \ref{hilbert:ume}, the operator  $\lambda_2(\mu)$ is then  uniformly ergodic. On the other hand, $1\in \sigma(\lambda_p(\mu))$. Since   isolated points of the spectrum are necessarily in the range of the Fourier-Stieltjes transform (see \cite[Lemma 2.2]{zafr73}) and $1\notin
\overline{\widehat{\mu}(\widehat{G})}$, we deduce from Corollary \ref{hilbert:ume} that $\lambda_p(\mu)$ is not uniformly mean ergodic.
  \end{proof}

\section{Positive  measures: mean ergodicity}
\label{mepositive}
This section is devoted to make clear that the ergodic behaviour of the operators $\lambda_p(\mu)$ is simpler when $\mu$ is positive.
\subsection{Positive measures supported in an  amenable subgroup: reflexive case}
% \label{sec:posME}
%\section{Positive measures}\label{sec:pos}

In this section we analyze mean ergodicity of convolution operators induced by positive measures whose support is contained in an amenable subgroup. The set of techniques at reach for this case is much richer and leads to conclusive results.

 \begin{theorem}
  \label{pb=me:amen}
   Let $G$ be a locally compact group and let $\mu\in M(G)$ be a \emph{positive} measure  with $H_\mu$ amenable. If $1<p<\infty$, the following assertions are  equivalent:
  \begin{enumerate}
    \item\label{11} $\lambda_p(\mu)$ is power bounded.
        \item\label{22} $\lambda_p(\mu)$ is mean ergodic.
\item\label{22b} $\lambda_p(\mu)$ is weakly mean ergodic.
      \item\label{44} $\lambda_p(\mu)$ is Ces\`aro bounded.
      \item\label{33} $\|\mu\| \leq 1$.
      \item\label{66} $\mu$ is vague-ergodic.
        \end{enumerate}
\end{theorem}
\begin{proof} By Corollary  \ref{adapted} we can restrict ourselves to the case when $G$ itself is amenable. The Mean Ergodic Theorem (Corollary \ref{MEth})  proves that \ref{11} implies \ref{22}. This one obviously implies  \ref{22b} and, by Banach-Steinhaus, \ref{22b} implies \ref{44}.
 Since $\norm{\mu}=r(\lambda_p(\mu))$, by Theorem \ref{cornormamen}, Corollary \ref{cbradio}  shows that \ref{44} implies \ref{33}. Since $r(\lambda_p(\mu))=\norm{\mu}$ implies that $\|\lambda_p(\mu)\|=\|\mu\|$,  Statement \ref{33} certainly implies Statement \ref{11}. Hence, statements \ref{11}--\ref{33} are all equivalent.

 Finally,  \ref{33} implies \ref{66},  by Theorem \ref{grenander}, and \ref{66} implies \ref{22b} by Theorem \ref{allp}.
 \end{proof}
 \begin{remark}
   Theorem \ref{pb=me:amen} does not hold if $H_\mu$ is not amenable, see Example \ref{ex:free}.
 \end{remark}
 \subsection{Positive measures: mean ergodicity in $L^1(G)$}
 Mean ergodicity of $\lambda_1(\mu)$ is a much more restrictive condition, as we now see. Here we are not assuming conditions a priori on $H_\mu$. First of all, we observe that we can reduce our study to probability measures.

 \begin{proposition}
 \label{reduccionprobabilidad}
 Let $G$ be a locally compact group and let $\mu\in M(G)$ be positive.
 \begin{enumerate}
 \item If $\|\mu\|<1$ then $(\lambda_1(\mu^n))_n$ is norm convergent to 0, and then $\lambda_1(\mu)$ is uniformly mean ergodic.
 \item If $\|\mu\|>1$ then $\lambda_1(\mu)$ is not mean ergodic.
 \end{enumerate}
 \end{proposition}
 \begin{proof}
 Observe that, for each $n\in\N$, $\norm{\lambda_1(\mu^n)}=\|\mu^n\|=\mu(G)^n$. Thus (1) is immediate and (2) follows from the unboundedness of the sequence  $(\|\lambda(\mu^n)\|/n)_n$ when $\mu(G)>1$.
 \end{proof}

\begin{theorem}
\label{lambda1}
Let $\mu$ be a probability measure on $G$. Then $\lambda_1(\mu)$ is mean ergodic if and only if $H_\mu$ is  compact.
\end{theorem}
\begin{proof}
If $H_\mu$ is compact, we only have to apply (2) implies (3)$^\prime$ of Theorem \ref{allp}.

Assume now that $H_\mu$ is not compact. By  Corollaries \ref{MEth} and  \ref{nocompact2}, $\lambda_2(\mu)$ is a mean ergodic operator without nonzero fixed points. As a consequence,  $\norm{\lambda_2(\mu)_{[n]}f}_2$ must converge to 0 for every  $f\in L^2(G)$.

Let  $f\in C_{00}(G)$, $f\geq 0$. By the above  $(\langle\lambda_{1}(\mu)_{[n]}f,g\rangle)_n$ converges to 0 for each $g\in C_{00}(G)$, i.e. $(\lambda_1(\mu)_{[n]}f)_n$ converges to 0 in the $\sigma(L^1(G),C_{00}(G))$-topology.
Since we are assuming that $\lambda_1(\mu)$ is mean ergodic, $\lambda_1(\mu)_{[n]}f$ must converge  in $L^1(G)$. But, as  the $\sigma(L^1,C_{00}(G))$-topology is Hausdorff and is weaker than the norm topology,  we conclude that, in fact,
\begin{equation}\label{to0} \lim_n \bnorm{\lambda_{1}(\mu)_{[n]}f}_1=0.\end{equation}
This actually holds  for every for every $f\in L^1(G)$ but we will only need that fact for some $f\in C_{00}(G)$, $f\geq 0$, $f\neq 0$.

A simple application of  Fubini's theorem shows that, for every $f\in L^1(G)$,
\[\int (\mu\ast f)(x) \dmg(x)=\int f(x)\dmg(x)\]
 %and, hence,  that, for every $f\in L^1(G)$, $f\geq 0$,
%\[\bnorm{\mu_{[n]}\ast f}_1=\bnorm{f}_1.\]
Then, for any $f\in C_{00}(G)$, $f\geq 0$,
\begin{align*}
  \bnorm{f}_1&=\bnorm{\mu_{[n]}\ast f}_1=\bnorm{ \lambda_1(\mu)_{[n]}f}_1,
\end{align*}
If we let $n$ go to infinity and apply \eqref{to0} we reach a contradiction  as soon as $f\neq 0$.
\end{proof}
    \begin{remark}\label{ebc}
  As can be remarked in the proof of Theorem \ref{lambda1}, the reason for the failure of  $\lambda_1(\mu)$ to be mean ergodic lies in  its action against positive functions. A way to avoid these functions is to consider the  subspace $L_0^1(G)=\left\{f\in L^1(G)\colon \int f \dmg=0\right\}$. Since this subspace  is invariant under the action of $\lambda_1$, one can consider the operator $\restr{\lambda_1^0(\mu)}{L_0^1(G)}$ and study  its mean ergodicity. This was done by Rosenblatt \cite{rose81} who defined a measure $\mu \in M(G)$ to be \emph{ergodic by convolutions} if $\lambda_1^0(\mu)_{[n]}$ converges to 0. In that same paper, Rosenblatt proves that a locally compact group contains a probability measure that is  ergodic by convolutions if and only if the group is $\sigma$-compact  and amenable. When $G$ is compact, the It\^o-Kawada Theorem proves that  a  probability measure $\mu\in M(G)$  is ergodic by convolutions if and only if $H_\mu=G$, and  the Choquet-Deny theorem implies that the same assertion is true when $G$ is Abelian \cite[Theorems 1.4 and 1.5]{rose81}, see also \cite{cuny03} and  \cite{jawo08} and the references therein for more on this property. It follows from the facts just  collected that every adapted (that is, with  $H_\mu=G$) probability measure  supported in a noncompact abelian group satisfies that $\lambda_1^0(\mu)$ is mean ergodic, while  $\lambda^1(\mu)$ is not.  On the other hand,  if $G$ is compact and Abelian, and $\mu\in M(G)$ is a probability measure which is not adapted, then $\lambda_1(\mu)$ is mean ergodic but $\lambda_1^0(\mu)$ is not. Neither  concept is  therefore  stronger than the other.

 %  With the results of this paper it is easy to prove that, for a compact group $G$,  two necessary and sufficient conditions for a measure $\mu\in M(G)$ with $\norm{\mu}\leq 1$ to be %ergodic by convolutions are (1) that $\mu$ is vague ergodic and the sequence $\mu_{[n]}$ converges to the Haar measure of $G$ and  (2) that  $G=H_\mu$. For probability %measures, this latter characterization follows from the It\^o-Kawada theorem, see \cite[Theorem 1.5]{rose81}.
\end{remark}
We can now complete the panorama of Proposition \ref{allp}, Theorem \ref{pb=me:amen} and Theorem \ref{lambda1}.
\begin{theorem}
  \label{ec0}
Let  $G$ be a  locally compact group and let  $\mu\in M(G)$ be a \emph{positive} measure with $H_\mu$ compact. For $1\leq p<\infty$, the following assertions are equivalent:
\begin{enumerate}
\item $\mu$ is vague-ergodic.
\item  $\mu$ is Ces\`aro bounded.
\item $\|\mu\|\leq 1$.
\item
$\lambda_p(\mu)$ is mean ergodic.
%\item   $\mu$ is Ces\`aro bounded.
        %        \item $\lambda_1(\mu)$ is mean ergodic.
\end{enumerate}
      \end{theorem}\
\begin{proof} That
Statement (1) implies Statement (2)  is a simple consequence of the Banach-Steinhaus theorem.
Next, (2)  is the same  as
Ces\`aro boundedness of  $\lambda_1(\mu)$ (by  (3) of Theorem \ref{young}) and the latter implies, by  Proposition \ref{cbradio}, that
 $r(\lambda_1(\mu))\leq 1$.   Since,  by  Proposition \ref{cornormamen},  this implies  $\norm{\mu}\leq 1$, we see that Statement (2) implies Statement (3). (3) implies (4) by Theorem \ref{pb=me:amen} if $p>1$ and by Theorem \ref{lambda1} for $p=1$. Finally (4) implies (1) follows from Proposition \ref{allp}.
\end{proof}
%\begin{remark}
% The equality $\norm{\mu}=r(\lambda_p(\mu))$ is in a sense exclusive to amenable groups: according to Th\`eor\`eme 5 of \cite{bergchrist74}, if $\mu \in M^+(G)$, $r(\lambda_2(\mu))=\|\mu\|_{_{M(G)}}$ if and only if $H_\mu$ is amenable.
%\end{remark}
 \section{Positive measures: uniform mean ergodicity    }\label{sec:posUME}
If $\mu$ is positive, the equality $\norm{\lambda_p(\mu)}=\norm{\mu}$ makes Theorem \ref{compactume} into the following somewhat cleaner characterization.
\begin{theorem}
  \label{reflexive}
  Let $G$ be a locally compact group and let $\mu$ be a positive operator-normal measure with $H_\mu$ amenable.
The following assertions are equivalent for $1\leq p\leq \infty$.
\begin{enumerate}
\item  $\lambda_p(\mu)$ is uniformly mean  ergodic.
\item  1 is not an accumulation point of $\sigma(\lambda_p(\mu))$, \emph{and}  $ \norm{\mu}\leq 1$.
    \end{enumerate}
  \end{theorem}
 Taking    Theorem \ref{lambda1}  into account, Theorem \ref{reflexive} yields the following Corollary. It applies to every probability measure in an Abelian group, although in that case a simpler proof using duality theory can be applied to every measure of norm 1.
  \begin{corollary}
Let $G$ be a locally compact  group and let $\mu$ be an operator normal probability measure such that $H_\mu$ is amenable and not compact. Then $1$ is an accumulation point of $\sigma(\lambda_1(\mu))$ $(=\sigma(\mu))$.
   \end{corollary}

  Theorem \ref{reflexive} leads to  a complete characterization  in the non reflexive case that is  valid  for, at least, all Abelian $G$.
  \begin{theorem}
  \label{nonreflexive}
  Let $G$ be a locally compact group and let $\mu\in M(G)$ be normal and positive. The following are equivalent
  \begin{enumerate}
  \item $\lambda_1(\mu)$ is uniformly mean ergodic.
   \item $\lambda_\infty(\mu)$ is uniformly mean ergodic.
   \item Either $\|\mu\| <1$,  or  $\norm{\mu}=1$,  $H_\mu$ is compact and 1 is not an accumulation point of $\sigma(\mu)$.
   \item $\lambda_\infty(\mu)$ is mean ergodic.
  \end{enumerate}
  \end{theorem}
  \begin{proof}
  Statement (2) is equivalent to Statement (4) by Theorem \ref{Lotz85}.  The remaining equivalences follow from   Theorem \ref{compactume} and Theorem \ref{lambda1},  observing that  $\lambda_1(\mu)^*=\lambda_\infty(\mu^*)$, $H_\mu=H_{\mu^*}$ and $\sigma(\mu)=\sigma(\mu^*)$ and taking into account that,  for a bounded linear operator on a Banach space $X$, the uniform ergodicity of $T$ is equivalent to that of $T^*$.
   \end{proof}

   \begin{remark}  When $G$ is compact, Theorem 5.6 of \cite{brownkarawill82} can be used to find examples of positive measures with $\lambda_2(\mu)$ uniformly mean ergodic for which $\lambda_1(\mu)$ is not. The measures obtained in \cite[Theorem 5.6]{brownkarawill82} are positive, belong to $M_0(G)$ and have the property that  $\sigma(\lambda_1(\mu))$ is the whole unit disk (see Lemma 4.1 of \cite{brownkarawill82} for this). For any such $\mu$, $\lambda_2(\mu)$ is uniformly mean ergodic, Corollary \ref{ume:ab:compact}, yet $\lambda_1(\mu) $ is not, Theorem \ref{reflexive}. Note that $\lambda_1(\mu)$ is mean ergodic by Theorem \ref{lambda1}.
    \end{remark}

 \subsection{Positive measures with noncompact support}

 When the support of a probability measure $\mu$ is not contained in a compact subgroup of $G$, $\lambda_p(\mu)$ does not have fixed points and its  dynamic behaviour   is much simpler.

\begin{theorem}\label{ume=nonamen:prob}
  Let $G$ be a locally compact group and let $\mu\in M(G)$ be a probability   measure  with $H_\mu$  noncompact. If  $1<p<\infty$ , the following assertions are  equivalent.
 \begin{enumerate}
   \item $\lambda_p(\mu)$ is uniformly mean ergodic.
   \item $r(\lambda_p(\mu))<1$.
      %\item $\norm{\mu}<1$.
   \item $(\lambda_p(\mu^n))_n$ is norm convergent to 0.
       \item $H_\mu$ is not amenable.
 \end{enumerate}
\end{theorem}
 \begin{proof}
 According to  Corollary \ref{adapted}, we can  assume along this proof that $G=H_\mu$. By Theorem \ref{cornormamen}, this immediately shows that Statements (2) and (4) are equivalent .

As for the remaining equivalences, only  that  (1) implies (4) needs proof. Assume to that end that  $\lambda_p(\mu)$ is uniformly mean ergodic.
Taking into account that, by Proposition \ref{con+lin:mu},  $1\notin \sigma(\lambda_p(\mu))$  Theorem \ref{cornormamen}, applies then to show that (4) holds.
\end{proof}
For convolution operators of norm at most one induced by measures  $\mu=f\dmg$ for some positive $f\in L^1(G)$, uniform mean ergodicity can be neatly characterized for all $p$, $1\leq p<\infty$.
\begin{theorem}\label{l1caracspect}
Let  $G$ be a locally compact group and  let $f\geq 0$, $f\in L^1(G)$, with $\norm{f}\leq 1$. Then $\lambda_p(f)$,  $1\leq p<\infty$, is  uniformly mean ergodic if and only if either $\supp(f)$ is contained in a compact subgroup of $G$ or  $r(\lambda_p(f))<1$.
\end{theorem}
\begin{proof}
After Theorems \ref{ume=nonamen:prob}  and \ref{lambda1}, it only remains to prove that $\lambda_p(f)$ is uniformly mean ergodic when the support of $f$ is contained in a compact subgroup of $G$.

The operator $\lambda^{H_f}_p(f)$ obtained by regarding $\lambda_p(f) $ as an operator on  $L^p(H_f)$,   is mean ergodic. This follows from Theorem \ref{lambda1} when $p=1$ and from  the Mean Ergodic Theorem, Corollary \ref{MEth} in the reflexive case $p>1$.

Since convolution with  a function in  $L^1(H_f)$ defines  a compact operator on $L^p(H_f)$ (see, e.g.,  \cite[Exercise 10.4.2]{dalesetal}) and, for power bounded compact operators,  mean ergodicity is equivalent to uniform mean ergodicity, \ref{compactme} of Theorem \ref{norms}, we have that  $\lambda^{H_f}_p(f)$ is a uniformly mean ergodic operator.
Corollary \ref{adapted} then shows that  $\lambda_p^G(f)$ is uniformly mean ergodic as well.
\end{proof}
The simple example in Remark  \ref{r4.13} shows that when $\mu\notin L^1(G)$, $H_\mu$ compact does not necessarily imply that $\lambda_p(\mu)$ is
 a uniformly mean ergodic operator.

%\subsection{Further characterizations}
The characterization of Theorem \ref{ume=nonamen:prob} leaves room for a convolution operator  $\lambda_p(\mu)$, with $\mu $ positive  and $H_\mu$ noncompact, to be uniformly  mean ergodic and yet $r(\lambda_p(\mu))=1$.  We see below that this cannot happen when $H_\mu$ is amenable.
 \begin{corollary}\label{ume=<1:amen}
 Let $G$ be a locally compact group and let $\mu\in M(G)$ be a \emph{positive} measure  with $H_\mu$ amenable and noncompact. If $1< p<\infty$, the following assertions are  equivalent:
 \begin{enumerate}
   \item $\lambda_p(\mu)$ is uniformly mean ergodic.
   \item  $\norm{\mu}<1$.
   \item $\lambda_p(\mu^n)$ is norm convergent to 0.
       \end{enumerate}
 \end{corollary}
 \begin{proof} Only that (1) implies (2) requires proof. The equality  $\|\lambda_p(\mu)\|=\|\mu\|$ (Theorem \ref{cornormamen}) permits us to proceed as in the proof of Proposition \ref{reduccionprobabilidad} to get $\norm{\mu}<1$.
  \end{proof}
  \begin{remark}
  \label{cuadrados3}
  The above Corollary can be used to prove the converse of Corollary \ref{cuadrados} for positive measures on amenable   groups that do not contain nontrivial compact subgroups (as, e.g., $G=\R$ or $G=\Z$).
  For Abelian $G$, more is true. In that case, it is not difficult to prove (relying on the positive-definiteness of $\widehat{\mu}$) that $-1$ is an accumulation point of $\widehat{\mu}(\widehat{G})$ if and only if $1$ is. It follows that for a probability measure $\mu \in M(G)$, $\lambda_2(\mu)$ is uniformly mean ergodic if and only if $\lambda_2(\mu^2)$ is. Positivity of $\mu$ is essential here, see Remark \ref{cuadrados2}.
  \end{remark}
  %\subsection{Positive measures with compact support}
%  In this subsection we deal mainly with the nonreflexive case, i.e. convolution operators defined on $L^1(G)$ and in $L^\infty(G)$. However, we state first a direct consequence for the reflexive case, which follows directly from Theorem \ref{compactume} and Theorem \ref{ec0}. This can be applied to every positive measure supported in an abelian compact group.

\section{Tracing limits}\label{sec:trac}

  \subsection{Counterexamples for nonpositive measures}
%The main   results of the preceding two sections depend in an essential way on the measure being positive and provide the same results for  convolution operators on $L^p$, regardless of the value of $p$,  $1<p<\infty$. The question that  arises is whether   a   measure such that $\lambda_p(\mu)$ is uniformly mean ergodic while $\lambda^q(\mu)$, $1<p\neq q<\infty$, is not, can be found.

 When $G$ is abelian and $\mu$ positive, Theorem \ref{pb=me:amen} and Theorem \ref{lambda1} in Section \ref{mepositive}  characterize  the mean ergodicity of $\lambda_p(\mu)$, for $1\leq p<\infty$. Also under the same assumptions,  Theorem \ref{reflexive} and Theorem \ref{nonreflexive} in Section \ref{sec:posUME}  give a complete characterization of the uniform mean ergodicity of $\lambda_p(\mu)$, $1\leq p\leq \infty$. In this section we give examples showing that these characterizations are not longer true when $\mu$ is not required to be positive.

Since our counterexamples will be Abelian, we recall from subsection \ref{sec:2.3} that, for Abelian  $G$, $\lambda_2(\mu)$ is unitarily equivalent to the multiplication operator by $\widehat{\mu}$ on $L^2(\widehat{G})$. As a consequence, the spectrum of $\lambda_2(\mu)$ is exactly $\overline{\widehat{\mu}(\widehat{G})}$.

%Our first examples are  very simple  but it  already show the failure,  respectively, Proposition \ref{con+lin:mu} and of Theorems   \ref{ume=nonamen:prob} and  \ref{ume=<1:amen} in %the absence of positivity.
%\begin{example}\label{ex:spmeas}
%A discrete  group $G$ and a measure $\mu\in
%\end{example}

The example below shows that Theorem \ref{ume=nonamen:prob}  fails if $\mu$ is not assumed to be positive, even in the Hilbert case, i.e., that $\lambda_2(\mu)$ can be  uniformly mean ergodic even if $\|\mu\|=1$.
 \begin{example}\label{ex:radnonp}
 A measure $\mu\in M(\Z)$ whose support generates $\Z$,  $r(\lambda_2(\mu))=1$ and yet $\lambda_2(\mu)$ is  uniformly mean ergodic. Moreover $(\lambda_2(\mu^n))_n$ is not norm convergent.
 \end{example}
 \begin{proof}
 Take   $\mu=(1/2)(\delta_1-\delta_2)   $.
 By Corollary \ref{hilbert:ume},  $\lambda_2(\mu)$ will be uniformly mean ergodic if and only if its spectrum is contained in the unit disc and does not contain $1$ as an accumulation point.

  In this case, for every $0\leq t< 2\pi$,
   \[\widehat{\mu}(e^{it})=\frac{1}{2}\left(
 e^{it}-e^{2it}\right)=\frac{1}{2}e^{it}\left(1-e^{it}\right).\]
 For $\widehat{\mu}(e^{it})=1$ one needs that $|1-e^{it}|=2$ and this only happens when $t=\pi$, but in that case $\widehat{\mu}(e^{it})=-1$. Hence $r(\lambda_2(\mu))=1$ but $1\notin \sigma(\lambda_2(\mu))$.

  Since $\widehat{\mu}^n(e^{it})=(-1)^n$ for each $n\in\N$ and $\norm{\lambda_2(\mu^n)}
  =r(\lambda_2(\mu^n))=r(\lambda_2(\mu))^n=1$, $(\lambda_2(\mu^n))_n$ cannot be  norm convergent to 0.  Observe that $\lambda_2(\mu)$ does not have non null fixed points by Theorem \ref{nocompact2}. We conclude that $(\lambda_2(\mu^n))_n$ cannot converge in norm.
   \end{proof}
   \begin{remark}
  \label{cuadrados2}
  In the above example $1\in \sigma(\lambda_2(\mu^2))$, hence $\lambda_2(\mu^2)$ is not uniformly mean ergodic. This shows that, unlike the positive case, the converse to Corollary 4.2  is not  true when $\mu$ is not positive.
  %Observe also that in this example $(\lambda_{[n]})$ is convergent to 0 by Proposition \ref{nocompactme}.
    \end{remark}
The  following example of a measure $\mu$ with $\norm{\mu}>1$, $H_\mu=\Z$   and $\lambda_p(\mu) $ uniformly mean ergodic for every $p$, reveals that positivity of $\mu$ is not a disposable condition in Proposition \ref{reduccionprobabilidad}, Theorem \ref{ec0},  Theorem  \ref{nonreflexive} or Corollary \ref{ume=<1:amen}.

 %When $\mu\in M(G)$ is not positive,  mean ergodicity of $\lambda_1(\mu)$ does not force $H_\mu $ to be compact, cf. Theorem \ref{ec0}. This is  shown in the following example. It also shows  that $\lambda_p(\mu)$ can be uniformly mean ergodic while $\norm{\mu}>1$ and $1\in \sigma(\lambda_2(\mu))$.
%
%
%Next example shows that the statement (3) is not equivalent to the other statements in Theorem \ref{nonreflexive} if $\mu$ is not assumed to be positive. Both $H_\mu$ non compact and $\|\mu\|>1$ provide an uniformly mean ergodic operator $\lambda_1(\mu)$. Also the example shows
% that the hypothesis of positivity of $\mu$ cannot be removed on Corollary . Neither Proposition  (c) is valid for non positive measures.
% % \begin{proposition}
%Let  $\mu\in M(\Z)(=l_1(\Z))$, $H_\mu\neq \{0\}$. The operator $\lambda_p(\mu)$ is mean ergodic if and only if $(\mu_{[n]} )$ is norm convergent to $0$.
%\end{proposition}
%
%\begin{proof}
%Since $\|\lambda_1(\mu)\|=\|\mu\|$ and $\delta\in l_1(\Z)$, if $\lambda_1(\mu_{[n]})(u)$ is convergent  for each $u\in l_1(\Z)$ then $\mu_{[n]}$ is convergent in $l_1(\Z)$. The convergence must be to a idempotent measure $\mu_c$. This forces $\mu_c=0$.  The converse is trivial since $\|\lambda_1(\mu)\|=\|\mu\|$.
%\end{proof}
%
%\begin{corollary}
%If $\mu\in l_1(\Z)$ and $\|\mu^k\|<1$ for some $k>1$ then  $\lambda_p(\mu^n)$ is norm convergent to 0 for each $1\leq p\leq \infty$.
%\end{corollary}

\begin{example}
\label{7.3}
Let $\mu=\frac13(\delta_1+\delta_0-\delta_{-1})\in M(\Z)(=l_1(\Z))$. For each $1< t<3/\sqrt{7}$, $\norm{t\mu}>1$ and  $\|\lambda_p(t\mu)^n)\|$ is  convergent to 0 and, hence, $\lambda_p(t\mu)$  is uniformly mean ergodic, for $1\leq p\leq \infty$.
\end{example}
\begin{proof}
We only have to observe that $\norm{\mu^2}=7/9$. As a consequence, as long as $t<3/\sqrt{7}$, $\lim_n\norm{t\mu^n}=0$.
\end{proof}

  \subsection{Uniformly mean ergodic convolution operators induced by  positive measures with \emph{large} support}
  We have seen in Theorem \ref{hilbert} that for an operator-normal measure $\mu\in M(G)$,   $\lambda_2(\mu)$ is  mean ergodic if and only if $\norm{\lambda_2(\mu)}\leq 1$. The same equivalence is shown to hold when $\mu$ is positive and supported in an amenable subgroup of $G$, Theorem \ref{pb=me:amen}, this time for $\lambda_p(\mu)$ and  $1<p<\infty$. In this latter case  mean ergodicity of $\lambda_p(\mu)$ is, in addition,   equivalent to vague-ergodicity of $\mu$.

  Here we show that,  when the measure is positive but not operator-normal and $H_\mu$ is not amenable, none of these equivalences remains true.  Our examples will consist on convolution operators defined by finitely supported measures in discrete groups.
We first collect some facts on  $\norm{\lambda_2(\mu)}$ for $\mu=\frac{1}{|S|}\sum_{s\in S} \delta_s$ with $S\subseteq G$.

Recall that elements of the free group $F(X)$ on the set of generators $X$ can be described uniquely as  words of the form $w=x_1^{\epsilon_1}\cdots x_2^{\epsilon_2}\cdot \cdots\cdot x_n^{\epsilon_n}$ with $x_i\in X$, $\epsilon_i=\pm 1$ and  $\epsilon_i=\epsilon_j$ whenever $x_i=x_j$. The length of $w$ is then  the minimum number of terms in such an expression of $w$. When all the exponents are positive, then $w$ belongs to the semigroup generated by $X$.
  \begin{theorem}\label{opnorms}
  Let $G$ be a discrete group and let $S\subseteq G$ with $|S|=n$. Consider
  $\mu=\sum_{s\in S} \delta_s\in M(G)$.
  %$\mu=\frac{1}{|S|}\sum_{s\in S} \delta_s\in M^+(G)$.
   Then:
  \begin{enumerate}
    \item (Particular case of \cite[Theorem 18.3]{pier}) $r(\lambda_2(\mu))=n$ if and only if $\<{S}$ is amenable. If $S$ contains the identity, $\norm{\lambda_2(\mu)}=n$ if and only if $\<{S}$ is amenable.
        \item (\cite[Theorem IV.K]{akemostr76}) If $S$ generates a free group:
         $\norm{\lambda_2(\mu)}=2\sqrt{n-1}$.
    \item (Haagerup inequality, \cite{haag79}) If   $G=F(X)$ is a free group and $S$ consists of  words of length $n$ on $X$, then: ${\displaystyle\norm{\lambda_2(\mu)}_{\mbld}\leq (n+1) \norm{\mu}_{l_2}}$.
        \item  (strong Haagerup inequality, \cite{kempspei07})  If   $G$ is a free group and $S$ consists  of words of length $n$ that are  in the semigroup generated by $X$, then ${\displaystyle\norm{\lambda_2(\mu)}_{\ll(L^2(G))}\leq e \sqrt{n+1}\norm{\mu}_{\ell^2}}$.
    \end{enumerate}  \end{theorem}

  This is the promised counterexample to     Theorem  \ref{pb=me:amen} for positive measures   which are not supported in amenable subgroups. This example also shows that the hypothesis  of normality of the measure  cannot be dropped to get the necessity of $\|\lambda_2(\mu)\|\leq 1$ when $\lambda_2(\mu)$ is uniformly mean ergodic in Theorem \ref{hilbert}, even if  $\mu$ is positive.
    \begin{example}\label{ex:free}
 Let $G$  be the free group on three generators  $G=F(x_1,x_2,x_3)$. There is a   finitely supported positive measure $\nu\in M(G)$ such that:
 \begin{enumerate}
   \item  $ \norm{\nu}>\norm{\lambda_2(\nu)}>1$.
   \item $\lambda_2(\nu)$ is uniformly mean ergodic.
   \item $\nu$ is not vague-ergodic.
     \end{enumerate}
    \label{ex:spec<norm}\end{example}
    \begin{proof}
 Let $\mu=(\delta_{x_1}+\delta_{x_2}+\delta_{x_3})$  and define, and for $\frac{1}{\sqrt{8}}<r<\frac{1}{\sqrt{3}}$, $\nu=r \mu$.

 %
%    \[\mu_2=\frac{1}{\sqrt{2}}\left(\delta_{x_1}+\delta_{x_2}\right),\quad
%    \mbox{ and  } \mu_3=\frac{1}{\sqrt{3}}\left(\delta_{x_1}+\delta_{x_2}+\delta_{x_3}
%    \right).\]
%    Then:
%    \begin{align*}
%      1=\norm{\mu_2}_2 & =r(\lambda_2(\mu_1))<\norm{\lambda_2(\mu_1)}_{\mbld}=\norm{\mu_1}_{_{M(G)}}=\sqrt{2}. \\
%  1=\norm{\mu_3}_2& =r(\lambda_2(\mu_2))     <\norm{\lambda_2(\mu_2)}_{\mbld}=\sqrt{\frac{8}{3}}<\norm{\mu_1}_{_{M(G)}}=\sqrt{3}.
%  \end{align*}
% If $1<\alpha_1<\sqrt{2}$ and $1<\alpha_2<\sqrt{\frac{8}{3}}$, the measures $\nu_i=\frac{1}{\alpha_i}\lambda_2(\mu_i)$ give  examples of:
% \begin{itemize}
%    \item Uniformly mean ergodic  operators $\lambda_2(\nu_i)$  produced by measures that are not mean ergodic.
%        \item Uniformly mean ergodic  operators $\lambda_2(\nu_i)$ with operator norm larger than 1, showing that Theorem \ref{pb=me:gen} is not true in general, even for positive measures.
%  \end{itemize} \end{example}
%  \begin{proof}

 Observe to begin with that, for each  $n\in \N$, the measures $\mu^{n}$    is precisely the  characteristic function of the set of all words of length  $n$ of the free semigroup generated by  $\{x_1,x_2,x_3\}$.

  Since, clearly, $\norm{\nu^{ n}}_{\ell^2}=(\sqrt{3}r)^n$ we deduce from the strong Haagerup inequality ((4) of Theorem \ref{opnorms}) that
  \[
  \norm{\lambda_2(\nu^{n})}\leq \left(\sqrt{3}r\right)^ne \sqrt{n+1}.
   \]
  The spectral radius formula together with the choice of $r$ yields
   \[
  \label{radmur} r(\lambda_2(\nu))\leq \sqrt{3}r<1\]
  Therefore, by \ref{r(T)} of Theorem \ref{norms}, $\lambda_2(\nu)$ is uniformly mean ergodic ($(\lambda_2(\nu^n))_n$ is even  convergent to 0)

  On the other hand,  (2) of Theorem \ref{opnorms} shows that
  \[
  \label{normao}
\norm{\lambda_2(\nu})= \sqrt{8} r>1.
  \]

 Finally, Since $\norm{\nu}=3r>\sqrt{8}r>1$ we get

\[\lim_n\frac{\norm{\nu^{n}}}{n}=\lim_n\frac{(3r)^n}{n}= \infty,\]
and we see that $\mu$ cannot be vague-ergodic.%
%  The equality $\norm{\lambda_2(\mu_2)}_{\mbld}=2\sqrt{\frac{2}{3}}$ follows from
%  The equality $\norm{\lambda_2(\mu_1)}_{\mbld}=\sqrt{2}$ is obtained if one takes into account that $\norm{\lambda_2(\delta_{x_1}+\delta_{x_2})}_{\mbld}=
%  \norm{\lambda_2(\delta_{e}+\delta_{x_2x_1^{-1}})}_{\mbld}$ and that
%  $\{e,x_2x_1^{-1}\}$ generates a cyclic infinite subgroup
% (i.e., $\Z$).
%
% For the remaining assertion, denote by  $G_{2,n}$  the set of words of length $n$ in the semigroup generated by  $x_1,x_2$. Then $\left|G_{2,n}\right|=2^n$ and,\footnote{{\tt ser\'ia cosa de usar s\'olo $\mu_2$ o $\mu_3$ har\'ia la notaci\'on m\'as sencilla}} since
% $\mu^{\ast ,n}=\frac{1}{\sqrt{2^n}}\sum_{w\in G_n}\delta_{w}$:
% \[ \Bigl\|\left(\nu_1\right)^{\ast n}\Bigr\|_{_{M(G)}}=
% \left(\frac{\sqrt{2}}{\alpha}\right)^{  n},\]
% while, exactly as in \eqref{powershaag}:
% \[\Bigl\|\left(\lambda_2(\nu_1)\right)^{\ast n}\Bigr\|_{_{\mbld}}
% \leq \frac{e\sqrt{n+1}}{\alpha^n}\]
% which implies $r\bigl(\lambda_2(\nu_1)\bigr)\leq \frac{1}{\alpha}<1$ and, hence, that $\lambda_2(\nu)$ is uniformly mean ergodic (Theorem \ref{norms}). The sameproof works for $\nu_2$.
   \end{proof}
\section{Open questions}
We remark that all our  examples  of mean ergodic operators are power bounded. It is natural to conjecture a positive answer to the following problem.

{\bf Problem 1.} Let $G$ be a locally compact group. Let  $\mu\in M(G)$ and $1<p<\infty$. Is it true that $\lambda_p(\mu)$ is power bounded whenever it is mean ergodic? What if $\mu\geq 0$? What if $H_\mu$ is amenable?

Examples \ref{7.3} and \ref{ex:free}  both  introduce {\em big} measures with $(\lambda_2(\mu^n))$ convergent in norm to 0, it seems also natural to ask if an  example in the spirit of \ref{ex:radnonp} can be obtained for  positive measures (necessarily on non-amenable groups).

{\bf Problem 2.}  Let $G$ be a free group (or any other nonamenable locally compact group). Is there any positive measure $\mu\in M(G)$ such that $\mu(G)>1$, $\lambda_2(\mu)$ is mean ergodic and $r(\lambda_2(\mu))=1$? If the answer is positive, could $\mu$ be taken  in such a way that also $\lambda_2(\mu)$ is uniformly mean ergodic?

Our last question refers to possible generalizations of Theorem \ref{pb=me:amen}  to nonpostive measures or nonamenable groups.

{\bf Problem 3.}  Is there a  locally compact group $G$ supporting a measure which is vague ergodic but $\lambda_p(\mu)$ is not mean ergodic? A positive answer would provide a convolution operator which is weakly mean ergodic but not mean ergodic.

  \section*{Appendix. Convolution operators $\lambda_p(\mu)$  as operators on $L^p(H_\mu)$: reducing to the support}\renewcommand{\thesection}{A}
%The  ergodic behavior of  a convolution operator $\lambda_p(\mu)$  for $\mu\in M(G)$  depends not as much on $G$ as on the smallest group on which $\mu$ manifests itself, the  subgroup generated by the support of $\mu$.}
%We consider now  measures $\mu\in M(G)$ supported on proper subgroups of $G$.
If $H$ is a closed subgroup of $G$ with $H\subseteq H_\mu$, convolving by $\mu$ can be seen both as an operator  on $L^p(H)$ and as an operator on $L^p(G)$. %To emphasize this difference, in the present section,  we  use $\lambda_p^H(\mu)$ and $\lam    bda_G^p(\mu)$ to denote each of these convolution operators.
Our aim is to show that  the ergodic behaviour of both operators is, as expected, the same.
The proofs of these results is rather technical and rely on several involved results of abstract harmonic analysis. We have therefore preferred to defer their  proof to this Appendix.

%\begin{notation}
%  Let $G$ be a locally compact group and let $\mu \in M(G)$. We denote by $H_\mu$ the closed subgroup of $G$ generated by $\supp \mu$.
%\end{notation}
%We use this section to check that the ergodic properties of the operator $\lambda_p(\mu)$ do not depend on whether we regard it as an operator on $\ll(L^p(G))$ or on $\ll(L^p(H_\mu))$. Although this behavior is expected, its verification  depends on some rather technical  nonelementary
%facts.

These facts are best described when  convolution operators by measures are seen in the wider frame of algebras of $p$-pseudomeasures. The algebra $PM_p(G)$ of $p$-pseudomeasures  is defined as the weak-operator closure of $\{\rho_p(f)\colon f\in L^1(G)\}$ in $\ll(L^p(G))$. $PM_p(G)$ is a Banach subalgebra of  $\ll(L^p(G))$ that contains $\rho_p(\mu)$ for every $\mu\in M(G)$.   It is easy to see that operators in $PM_p (G)$ commute with left translations.

As a Banach space, the algebra $PM_p(G)$ can be seen as the dual space of a function algebra $A_p(G)$ known as the Fig\`a-Talamanca Herz algebra.  We will not need to provide a precise defintion of  this algebra  here. It suffices to say that for any $f\in L^p(G)$ and $g\in L^{p^\prime}(G)$, with $1/p+1/p^{\prime}=1$, the convolution $\bar{f}\ast \check{g}\in A_p(G)$, where $\bar{f}(s)=\overline{f(s)}$ and $\check{f}(s)=f(s^{-1})$,   and that, for each $T\in PM_p(G)$,
$\langle T, \bar{f}\ast \check{g}\rangle=\langle Tf, g\rangle$, where the first bracket corresponds to the $(PM_p(G),A_p(G))$-duality and the second to the $(L^p(G),L^{\pp}(G))$-duality.

In our setting, it  would have been more natural to introduce   the algebra of pseudomeasures as the weak-operator closure of  $\{\lambda_p(f)\colon f\in L^1(G)\}$, as it is  often done in the literature. This would have produced a different but linearly isometric algebra. By technical reasons related to Theorem  \ref{fromhtog} we find it preferable to use the right-handed version here.

Since in this section we are going to see  the operator $\rho_p(\mu)$, $\mu\in M(H)$,   both as an  operator on  $L^p(G)$ and as   an operator on $L^p(H)$,  where $H$ is a subgroup of $G$, it will be convenient to use the notation $\rho^{\G}_p(\mu)$ and  $\rho^{\hh}_p(\mu)$ (or $\lambda^{\G}_p(\mu)$ and $\lambda^{\hh}_p(\mu)$ for the left-handed versions) to distinguish both cases.

The basic tool to explore  the relation between operators on $L^p(G)$ and $L^p(H)$ is the Mackey-Bruhat integration formula described in the next lemma.
%\begin{lemma}[Mackey-Bruhat integration formulaintegration formula. Remark 8.2.3 of \cite{reiter}]\label{WIF}
%Let $G$ be a   locally compact group and let $H$ be a closed subgroup of $G$. such that the  space of left cosets  $G/H$ carries an invariant measure $\mgh$. Then:
%\[\int f(x)\dmg(x)=\int_{G/H}\left(\int_H f(xt) \dmh(t)
%\right)\dmgh(\dot{x}),\]
%where $\dot{x}$ denotes the left coset $xH$.
% \end{lemma}
% Note that for $G/H$ to carry an invariant measure is enough that $H$ be normal or compact or that $G$ be unimodular (and that includes both the Abelian and %discrete case).
\begin{lemma}[Mackey-Bruhat  integration formula. Remark 8.2.3 of \cite{reiter}]\label{WIF}
Let $G$ be a   locally compact group and let $H$ be a closed subgroup of $G$. There is a quasi-invariant measure $\mgh$ on the  space of left cosets $G/H$  and a continuous strictly positive function $q\colon G\to \R$ such that:
\begin{align*}
&\frac{q(xh)}{q(x)}=\frac{\Delta_{\hh}(h)}{\Delta_{\G}(h)}, \mbox{ and }\\
&\int f(x)\dmg(x)=\int_{G/H}\left(\int_H\frac{f(xh)}{q(xh)}\dmh(h)
\right)\dmgh(\dot{x}),
\end{align*}
where $\dot{x}$ denotes the right coset $xH$.
If $H$ is a normal subgroup of $G$, $q(x)=1$ for every $x\in G$.
\end{lemma}
%
%In the following, given  $f\in L^p(G)$ and $x\in G$ we will   define  $f_{x}\in L^p(H)$ by \begin{align*}
% f_{x}(h)&=\frac{f(xh)}{q(xh)^{1/p}}, \mbox{ for any } h\in H.
% \end{align*}
%\begin{lemma}[Lemme 1 of \cite{loho80}, reminiscent of Lemma 1.1 of \cite{mack52}]\label{thick}
%%[Lemme 1 in the Appendice of \cite{loho80}
%Let $G$ be  a metrizable and separable  locally compact group  and let $H$ be a closed subgroup of $G$ Then:\begin{itemize}
%  \item There is a measurable function $\theta \colon G\to H$  such that $\theta(hx)=h\theta(x) $ for every $h\in G$ and $x\in G$.
%      \item If $V$ is a compact neighbourhood of the identity in $G$ and $f\in C_{00}(H)$, then the $f_\theta(x)=f\bigl(\theta(x)\bigr)\cdot \Cf_{HV}$  defines a function $f_\theta \in C_{00}(G)$ and there is $M>0$ such that\[\int_G f_\theta(x)\dmg(x)=M\int_H f(h)\dmh(h).\]
%\end{itemize}
%\end{lemma}
%\begin{corollary}
%  Let $G$ be a locally compact group, let $H$ be a closed subgroup of $G$ and let $\mu\in M(G)$ be supported in $H$. Then, for any  $1<p<\infty$, \[r\left(\lambda_p^{\hh}(\mu)\right)=
%  r\left(\lambda_p^{\G}(\mu)\right)
% .\]
%\end{corollary}
%\begin{proof}
%We only have to apply  (4) of
%Lemma \ref{pseduom} to the spectral radius formula.
%\end{proof}

\begin{lemma}\label{pseduom}
Let $G$ be a locally compact group, let $H$ be a closed subgroup of $G$ and let $1<p<\infty$. Then,
\begin{enumerate}
  \item\label{x} On bounded subsets of $PM_p(G)$ the $\sigma(PM_p(G),A_p(G))$-  and weak operator topologies coincide.
  \item\label{y} Restriction defines a linear surjective mapping  $R_H\colon A_p(G)\to A_p(H)$ such that for each  $h\in A_p(H)$ and $\varepsilon>0$ there is $g\in A_p(G)$ with $\norm{h}\leq \norm{g}\leq \norm{h}+\varepsilon$ and $R_H(g)=h$.
      \item\label{z} The adjoint $(R_H)^\ast \colon PM_p(H)\to PM_p(G)$  is a multiplicative linear isometry.
          \item\label{t} If $Q\in PM_p(H)$, $f\in L^p(G)$, $x \in G$  and $h\in H$,
    \[ (R_H)^\ast(Q)f(xh)=Q f_x(h) q^{1/p}(xh).\]
    \item $(R_H)^\ast(\rho_p^H(\mu))=\rho_p^G(\mu)$.
\end{enumerate}
\end{lemma}\begin{proof}
  The first item   follows, e.g., from Theorem 6 of \cite{deri11}. Items (2) and (3) can be deduced from Theorems A and   1 of \cite{herz73} and also from  \cite[Proposition 7.3.5 and Theorem 7.8.4]{deri11}.

  Item (4) follows from Theorem 7.8.4 of \cite{deri11} after noting that the operator $f(xh)\mapsto Q f_x(h) q^{1/p}(xh)$, $f\in L^p(G)$, coincides with the operator $i\colon PM_p(H)\to PM_p(G)$ introduced in Definition 7.2.7 \emph{loc. cit.} Since operators in $PM_p(H)$ commute with left translations, the formula in (4) defines this operator  unambiguously.

  Item (5) follows after applying (4) to $\rho_p^H(\mu)$ and any $f\in L^p(G)$, $x\in G$:
  \begin{align}\label{muG}
 ((R_H)^\ast(\rho^{\hh}_p(\mu)))f(xh)&=
 \rho_{\hh}(\mu)^p f_{x} (h)q^{1/p}(xh) \notag\\
 &=\int \Delta_{\hh}(u)^{1/p}\frac{ f(xhu)}{q(xhu)^{1/p}} q(xh)^{1/p} \dmu(u) \notag\\
 &=      \int \Delta_{\G}(u)^{1/p} f(xhu)\dmu(u)\\&=
 \rho^{\G}_p(\mu)f(xh).\notag
 \end{align}
\end{proof}

\begin{theorem}\label{fromhtog}
Let $G$ be  a  locally compact group,    $H$    a closed subgroup of $G$ and   $\mu\in M(G)$  with $\supp (\mu)\subseteq H$. Let as well  $1<p<\infty$.Then $\rho_H^p(\mu)$ is (uniformly, weakly) mean ergodic if and only if $\rho_G^p(\mu)$ is (uniformly,  weakly) mean ergodic.

%\begin{enumerate}
%\item

%%Assume that  $G/H$ carries an invariant measure $\mgh$.
%If $\lambda_p^H(\mu)$ is mean ergodic, then $\lambda_p^\G(\mu)$ is mean ergodic.
%    \item  Assume that  $G/H$ carries an invariant measure $\mgh$.  If $\lambda_p^H(\mu)$ is uniformly  mean ergodic, then $\lambda_p^\G(\mu)$ is uniformly  mean ergodic.
%\item Assume that $G$ is  metrizable and separable. If  $\lambda_p^\G(\mu)$ is mean ergodic, then $\lambda_p^H(\mu)$ is mean ergodic and if
%    $\lambda_p^\G(\mu)$ is uniformly mean ergodic if and only if $\lambda_p^H(\mu)$ is uniformly mean ergodic.
%\end{enumerate}
\end{theorem}
\begin{proof}

\emph{1. Uniform mean ergodicity.}\\
Since $(R_H)^\ast\colon PM_p(H)\to PM_p(G)$ is a multiplicative linear isometry and $(R_H)^\ast(\rho_p^H(\mu))=\rho_p^G(\mu)$, Lemma \ref{pseduom}, it is clear that $\rho_p^H(\mu)$ is uniformly mean ergodic if and only if $\rho_p^G(\mu)$ is uniformly mean ergodic.

\emph{2. Weak mean ergodicity.}\\
Lemma \ref{pseduom} and the Banach-Steinhaus theorem
imply that Condition (2) of Theorem \ref{weakyosida} holds for  $\rho_H^p(\mu)$ if and only if it holds  or  $\rho_G^p(\mu)$. It also follows that as soon as  either $\rho_H^p(\mu)$ or  $\rho_G^p(\mu)$ is weakly mean ergodic then  $\lambda_H^p(\mu)_{[n]}$, $\rho_G^p(\mu)_{[n]}$,
$\frac{1}{n}\rho_G(\mu^{n})$ and $\frac{1}{n}\rho_H(\mu^{n})$ will all  be bounded in the operator norm.

Suppose now that  $\frac{1}{n}\rho_H^p(\mu)^n$ converges   to 0 in the weak operator topology and let  $f\in L^p(G)$ and $g\in L^{p^\prime}(G)$,
then
\begin{align*}
  \frac{1}{n}\left\langle \rho_G^p(\mu)^nf,g\right\rangle&=\frac{1}{n}
  \left\langle \rho_G^p(\mu^{n}),\bar{ f}\ast \check{g} \right\rangle\\&=\frac{1}{n}
  \left\langle \rho_H^p(\mu^{n}), R_H(\bar{f}\ast \check{g})\right\rangle.
\end{align*}
So, since weak operator topology and $\sigma(PM_p(G),A_p(G))$ coincide on bounded sets of $PM_p(G) $ ((1) of Lemma \ref{pseduom}),    $\frac{1}{n}\rho_G^p(\mu^n)$ converges to 0 in the weak operator topology.

If, conversely, $\frac{1}{n}\rho_G^p(\mu)^n$  converges to 0  in the weak operator topology and $f\in L^p(H)$, $g\in L^{p^\prime}(H)$ we can consider
$u\in A_p(G)$ with $R_H(u)=\bar{f}\ast\check{g}$ and, noting again that $R^\ast(\rho_p^H(\mu))=\rho_p^G(\mu)$,
\begin{align*}
  \frac{1}{n}\left\langle \rho_H^p(\mu)^n f,g\right\rangle&=\frac{1}{n}
  \left\langle \rho_H^p(\mu^{n}), \bar{f}\ast \check{ g} \right\rangle\\&=\frac{1}{n}
  \left\langle \rho_H^p(\mu^{n}), R_H(u)\right\rangle\\
  &=\frac{1}{n}\left\langle \rho_G^p(\mu^{n}), u\right\rangle
\end{align*}
and we see that  $\frac{1}{n}\rho_H^p(\mu)^n$ converges to 0 in the  weak operator topology.

As both conditions of  Theorem \ref{weakyosida} hold precisely for $\rho_p^G(\mu)$ when they hold for $\rho_p^H(\mu)$, we conclude that  $\rho_p^G(\mu)$ is weakly mean ergodic if and only if $\rho_p^H(\mu)$ is.

\emph{3. Mean ergodicity.}\\
We are now going to use the norm topology version of Theorem \ref{weakyosida}. As above, Condition (2) will be satisfied for $\rho^{\G}_p(\mu)$ if and only if it is satisfied for  $\rho^{\hh}_p(\mu)$.

Let $\mu_n=\frac{\mu^{n}}{n}$. We will show that $\rho^{\G}_p(\mu_n)$ converges to 0 in the strong
operator topology if and only if so does $\rho^{\hh}_p(\mu_n)$.   Our approach  will   follow closely Chapter 7 of \cite{deri11}.

So, let us first assume  that $\rho^{\hh}_p(\mu_n)$ converges to $0 $ in the strong operator topology.
 %It is easy to check that the operator $P$ commutes with left translations.
%
% For each $f\in L^p(G)$ and $s\in G$ we  consider the function $q$ introduced in Theorem \ref{WIF} and define $f_{x}(h)\in L^p(H)$ by \begin{align*}
% %f^\ast_s(h)&= f(sh) \mbox{ and }\\
% f_{x}(h)&=\frac{f(xh)}{q(xh)^{1/p}}, \mbox{ for any } h\in H.
% \end{align*}
% This function is used to  lift  each operator $Q\in \ll(L^P(H))$ that commutes with left translations to an operator $Q^{\G}\in \ll(L^p(G))$ given by:
% \[ Q^{\G}f(xh)=Q f_x(h) q^{1/p}(xh),\]  for every $f\in L^p(G)$,  every $x\in G$ and $h\in H$.
% Note that $Q^{\G}$ is  well defined because $Q$ commutes with left translations. It follows from Lemma 7.1.11 in \cite{deri11} that $Q^{\G}$ actually belongs to $\mblp$.
%
% It is important to observe that, as a consequence of the properties of the function $q$, for any $\mu\in M(G)$, $f\in L^p(G)$, $x\in G$ and $h\in H$, one has that
Using the Mackey-Bruhat formula in Lemma \ref{WIF}, for $f\in L^p(G)$,
\begin{align}\label{MBrho}
  \norm{\rho_{\G}(\mu_n)f}^p_{_{L^p(G)}} & =
  %=\int_G \left|(\mu_n \ast_{\G} f )(s)\right|^p ds\\
  %= &
   \int_{G/H}\int_H \frac{\left|\rho_{\G}(\mu_n)f(sh)\right|^p }{q(sh)}\dmh(h)\dmgh(\dot{s}).
  \end{align}
  Putting $f_{s}(h)=\frac{f(sh)}{q(sh)^{1/p}}$ and applying the properties of the function $q$ to the inner integral in \eqref{MBrho}
  \begin{align*}
\int_H\frac{\left|(\rho_{\G}(\mu_n)f )(sh)\right|^p }{q(sh)}\dmh(h)&= \int_H \left|\int_H \frac{\Delta_{\G}(u)^{1/p} f(shu) }{q(sh)^{1/p}}d\mu_n(u)\right|^p\dmh(h)\\
&=\int_H \left|\int_H \frac{\Delta_{\G}(u)^{1/p} f(shu) \Delta_{\hh}(u)^{1/p} }{q(shu)^{1/p}\Delta_{\G}(u)^{1/p}}d\mu_n(u)\right|^p \dmh(h)
\\
&=\int_H \left|\int_H\frac{\Delta_{\hh}(u)^{1/p}  f(shu)  }{q(shu)^{1/p} }d\mu_n(u)\right|^p\dmh(h)
\\&=\norm{(\rho_{\hh}(\mu_n) f_s)}_{_{L^P(H)}}^p.
  \end{align*}

%\begin{align}\label{eq:MB1}
%  \norm{\rho^{\hh}_p\left(\mu_n\right)f}^p_{_{L^p(G)}} & =\int_G \left|(T_k)^{\G}f )(s)\right|^p ds \notag\\
%   &= \int_{G/H}\int_H \frac{\left|T_k^{\G}f(sh)\right|^p }{q(sh)}\dmh(h)\dmgh(\dot{s})
%  \\
%  &= \int_{G/H}\int_H \frac{\left|T_kf_s(h)q^{1/p}(sh)\right|^p }{q(sh)}\dmh(h)\dmgh(\dot{s})
%  \\&= \int_{G/H} \norm{T_k f_s}_{_{L^p(H)}}.
%  \notag %%
%  \end{align}
%% Applying the equality in \eqref{muG}
%%  \begin{align*}
%%\frac{\left|T_k^{\G}f(sh)\right|^p }{q(sh)}&= \left|\int_H \frac{\Delta_{\G}(u)^{1/p} f(shu) }{q(sh)^{1/p}}d\mu_n(u)\right|^p\\
%%&=\left|\int_H \frac{\Delta_{\G}(u)^{1/p} f(shu) \Delta_{\hh}(u)^{1/p} }{q(shu)^{1/p}\Delta_{\G}(u)^{1/p}}d\mu_n(u)\right|^p
%%\\
%%&=\left|\int_H\frac{\Delta_{\hh}(u)^{1/p}  f(shu)  }{q(shu)^{1/p} }d\mu_n(u)\right|^p
%%\\&\left|\int_H (\rho^{\hh}_p (\mu_n) f_s)(h)\right|^p.
%%  \end{align*}
%Continuing in \eqref{eq:MB1},
%\begin{align*}
%  \norm{\rho_{\G}(\mu_n)f}^p_{_{L^p(G)}} &=\int_{G/H}\norm{\rho^{\hh}_p(\mu_n)f_s}_{_{L^p(H)}}^p\dmgh(\dot{s}).
%\end{align*}
Since $\rho_{\hh}(\mu_n)$ converges to 0 in the strong operator topology, the sequence $\norm{\rho_{\hh}(\mu_n) f_s}_{_{L^p(H)}}^p$ converges to 0 for every $\dot{s}\in G/H$. By Lebesgue's dominated convergence theorem, applied to the integral in \eqref{MBrho}, the sequence $\norm{\rho_{\G}(\mu_n)f}{_{L^p(G)}}$ will  converge to 0 as long as we can see that  the functions $\norm{\rho_{\hh}(\mu_n) f_s}_{_{L^p(H)}}^p$ are dominated by some integrable function.
But, for each $\dot{s}\in G/H$, $\norm{\rho_{\hh}(\mu_n) f_s}_{_{L^p(H)}}^p\leq \norm{\rho_{\hh}(\mu_n)}^p\cdot  \norm{f_s}_{_{L^p(H)}}^p$, and the Bruhat-Mackey formula implies that $\dot{s}\mapsto \norm{f_s}_{_{L^p(H)}}^p$ is integrable with, precisely, \[\int_{G/H}\norm{f_s}_{_{L^p(H)}}^p\dmgh(\dot{s})=\norm{f}_{_{L^p(G)}}^p.\]
We conclude so that  $\rho^{\G}_p(\mu_n)f$ converges to $0$ in norm.

Assume now that $\rho^{\G}_p(\mu_n)$ converges to 0 and let $f\in L^p(H)$, $g\in L^{\pp}(H)$ with $\norm{g}_{_{L^{\pp}(H)}}\leq 1$. If we follow the proof of Theorem 7.3.2 of \cite{deri11}, we can find two functions $v_f\in L^p(G)$, $v_g\in L^{\pp}(G)$ such that (this is the top formula of page 115 {\it loc. cit.})
\begin{align*}
\norm{v_g}_{_{L^{\pp}(G)}}&\leq\norm{g}_{_{L^{\pp}(H)}}\leq 1 \quad \mbox{ and } \\
\left|\bang{\rho^{\hh}_p(\mu_n)f,g}\right|&\leq
\left|\bang{\rho^{\G}_p(\mu_n)v_f,v_g}\right|.
\end{align*}
It follows then that $\lim_{n\to \infty}\norm{(\rho^{\hh}_p(\mu_n)f}_{_{L^p(H)}}=0$.

\end{proof}
The equivalence between $\rho_p(\mu)$ and $\lambda_p(\mu)$ stated in Fact  \ref{intert} leads to the following Corollary.
\begin{corollary}\label{adapted}
Let $G$ be  a  locally compact group,    $H$    a closed subgroup of $G$ and   $\mu\in M(G)$  with $\supp (\mu)\subseteq H$. Let as well  $1<p<\infty$. Then:
\begin{enumerate}
\item $\|\lambda_H^p(\mu)\|=\|\lambda_G^p(\mu)\| $. Hence,
    \item $r\left(\lambda_p^{\hh}(\mu)\right)=
  r\left(\lambda_p^{\G}(\mu)\right)
 $.
%\lambda_H^p(\mu)$ is power bounded if and only if $\lambda_G^p(\mu)$ is.
\item $\lambda_H^p(\mu)$ is (uniformly, weakly) mean ergodic if and only if $\lambda_G^p(\mu)$ is.
\end{enumerate}
\end{corollary}
%\begin{corollary}\label{fromhtog2}
%Let $G$ be  a  locally compact group,    $H$    a closed subgroup of $G$ and   $\mu\in M(G)$  with $\supp (\mu)\subseteq H$. Let as well  $1<p<\infty$. Then $\lambda_H^p(\mu)$ is (uniformly, weakly) mean ergodic if and only if $\lambda_G^p(\mu)$ is (uniformly,  weakly) mean ergodic.
%\end{corollary}

\subsection*{Acknowledgements} We have profitted from discussions around these topics with several colleagues. We would like in particular to thank Antoine Derighetti for always being ready to answer our queries on how to relate $\lambda_p^{\hh}(\mu)$ and $\lambda^{\G}_p(\mu)$, to  Michael Lin for helping us to gain perspective on the probabilistic side of convolution operators, to Nico Spronk and Matthias Neufang for making us aware of \cite{mukh76}  and sharing with us some details of \cite{neufsalmiskalspro}, which  permitted us  avoid our own, unnecessary, proof of  Theorem \ref{nocompact2}, and to Przemys{\l}aw Ohrysko for pointing us towards the paper \cite{iga69} and hence setting us on the path to Proposition \ref{exl2diflp}.

\end{document}